\documentclass[11pt,a4paper]{amsart}
\usepackage{amsmath,caption,booktabs,lipsum}
\numberwithin{equation}{section}
\usepackage{amsthm}
\usepackage{pdfpages}
\usepackage{relsize}
\usepackage{amsfonts}
\usepackage{amssymb}
\usepackage{mathrsfs}
\usepackage{tikz}
\usepackage{caption}
\usepackage{environ}
\usepackage{elocalloc}
\usepackage{url}
\usepackage{caption}
\usepackage{graphicx}
\usepackage[backend=bibtex,doi=false,isbn=false,url=false]{biblatex}
\usepackage{CJKutf8}
\usepackage[normalem]{ulem}
\usepackage{xcolor}
\usepackage{etoolbox}
\usepackage{caption}
\usepackage{mathtools}
\usepackage{dcpic}
\usepackage{tikz-cd}
\usepackage[bottom]{footmisc}
\usepackage[hyperfootnotes=false]{hyperref}
\usepackage{enumitem}
\usepackage{stmaryrd}
\usetikzlibrary{positioning}
\usetikzlibrary{arrows,chains,positioning,scopes,quotes}
\usetikzlibrary{decorations.markings}

\newcommand{\supp}{\operatorname{Supp}}
\newcommand{\reg}{\operatorname{Reg}}
\newcommand{\ai}{\alpha}

\newcommand{\sing}{\operatorname{Sing}}

\newcommand{\inj}{\operatorname{Inj}}
\newcommand{\es}{\emptyset}

\newcommand{\mn}{\mathbf{MIN}}
\newcommand{\be}{\beta}
\newcommand{\Ga}{\Gamma}
\newcommand{\ga}{\gamma}

\newcommand{\de}{\delta}
\newcommand{\De}{\Delta}
\newcommand{\e}{\epsilon}
\newcommand{\lag}{\mathbf{LavGap}}
\newcommand{\hm}{\mathcal{H}}
\newcommand{\lam}{\lambda}

\newcommand{\om}{\omega}

\newcommand{\Om}{\Omega}
\newcommand{\si}{\sigma}
\newcommand{\Si}{\Sigma}

\newcommand{\rh}{\rho}
\newcommand{\ta}{\theta}

\newcommand{\cms}{\operatorname{comass}}

\newcommand\res{\mathop{\hbox{\vrule height 7pt width .3pt depth 0pt
			\vrule height .3pt width 5pt depth 0pt}}\nolimits}

\newcommand{\ms}{\mathbf{M}}

\newcommand{\cd}{\cdots}
\newcommand{\T}{\mathbf{T}}

\newcommand{\s}{\subset}

\newcommand{\cp}{^\complement}

\newcommand{\la}{\langle}
\newcommand{\ra}{\rangle}
\newcommand{\ran}{\mathbb{R}_{\operatorname{an}}}
\newcommand{\nog}[1]{\no{#1}^{\R}}

\newcommand{\no}[1]{\left\lVert#1\right\rVert}

\DeclarePairedDelimiter{\ri}{\la}{\ra}
\newcommand{\dvol}{\operatorname{dvol}}

\newcommand{\du}{^\ast}
\newcommand{\br}{\mathbf{B}_r}
\newcommand{\pf}{_\ast}
\newcommand{\spt}{\operatorname{Supp}}

\newcommand{\ka}{\kappa}

\newcommand{\m}{^{-1}}
\newcommand{\ts}{\otimes}

\newcommand{\pd}{\partial}

\newcommand{\na}{\nabla}

\newcommand{\N}{\mathbb{N}}
\newcommand{\R}{\mathbb{R}}
\newcommand{\Z}{\mathbb{Z}}

\newcommand{\Q}{\mathbb{Q}}

\newcommand{\eq}{\Leftrightarrow}

\makeatletter
\def\thm@space@setup{%
	\thm@preskip=0.2cm plus 0cm minus 0cm
	\thm@postskip=\thm@preskip 
}
\makeatother
\theoremstyle{plain}
\newtheorem{thm}{Theorem}
\newtheorem{rst}{Result}

\newtheorem{lem}{Lemma}[subsection]

\newtheorem{assump}{Assumption}[subsection]
\newtheorem{cor}{Corollary}[subsection]

\newtheorem{fact}{Fact}[subsection]
\newtheorem{defn}{Definition}[subsection]
\newtheorem{conj}{Conjecture}
\newtheorem{prob}{Problem}

\theoremstyle{definition}

\title[The Hasse Principle for Area-minimizing Submanifolds]{The Hasse Principle for Area-minimizing Submanifolds}
\author{Zhenhua Liu}
\dedicatory{Dedicated to Xunjing Wei}
\numberwithin{equation}{subsection}
\begin{document}
	\setlength{\abovedisplayskip}{5pt}
	\setlength{\belowdisplayskip}{5pt}
	\setlength{\abovedisplayshortskip}{5pt}
	\setlength{\belowdisplayshortskip}{5pt}
	\begin{abstract}
	The Hasse principle in number theory states that information about integral solutions to Diophantine equations can be pieced together from real solutions and solutions modulo prime powers. We show that an analogous Hasse principle holds for area-minimizing submanifolds: information about area-minimizing submanifolds in integral homology can be recovered from those in real homology and mod $n$ homology for all $n\in \mathbb{Z}_{\ge 2}$. As a consequence we answer several questions of Almgren, Morgan and White and prove: area-minimizing submanifolds are not generically calibrated, products of area-minimizing submanifolds are not generically area-minimizing, and classification of area-minimizing pairs of planes mod $n$ for $n\ge 4$. 
	\end{abstract}
	\maketitle\vspace{-3em}
	\section{Introduction}\label{hsintro}
	The Hasse principle \cite{HHhp,JSec} in number theory states that information about integral solutions to Diophantine equations can be pieced together from real solutions and all solutions modulo prime powers. It dates back to the Hasse-Minkowski theorem for quadratic forms \cite{JSca}, which says a quadratic form represents zero over the integers if and only if it represents zero over the reals and modulo prime powers. Another classical example is the original formulation of the Birch and Swinnerton-Dyer conjecture \cite{BSD}, where the  number of solutions modulo primes conjecturally gives ranks of rational solutions of elliptic curves. Though the Hasse principle does not hold in general \cite{EShp}, the guiding philosophy remains a pillar of modern number theory.
	
	An analytical analogue lies undetected in Riemannian geometry and geometric measure theory. The classical variational problem of finding area-minimizing representatives in integral homology classes is solved by Federer and Fleming (\cite{FF}):
	\begin{quote}
		\emph{Every integral homology class on a compact Riemannian manifold admits an area-minimizing representative.}
	\end{quote}
	Their minimizers live in the space of integral currents, i.e., the Whitney-flat closure of integer-coefficient polyhedral chains with themselves and their boundaries having finite area \cite{FF}.
	
	Parallel existence theorems hold for real homology and for homology with $\mathbb Z/n\mathbb Z$ coefficients for all $n\in \Z_{\ge 2}$, established respectively in \cite{HFrf} and \cite{WFfc,BWrc,BWdt,WZmod2}. In these cases, the set of representatives of homology classes is significantly enlarged and a priori there is no reason to expect any connection at all between area-minimizing representatives of homology classes in different coefficients.
	
	According to Ulam \cite{SUna}, Banach often remarked
	\begin{quote}
		\emph{Good mathematicians see analogies between theorems or theories; the very best ones see analogies between analogies.}
	\end{quote}
	Inspired by this maxim, we show the number theoretic Hasse principle has a counterpart in geometric measure theory, two areas seemingly worlds apart:	\begin{rst}\label{rstm}
		Area-minimizing representatives of integral homology classes can be recovered from those of real homology and mod $n$ homologies for $n\in \Z_{\ge 2}.$
	\end{rst}
The more precise statement of the above Result \ref{rstm} is Theorem \ref{thmv} in Section \ref{secmtc}. Analogous results also holds in the setting of area-minimizing cones and Plateau problems in Euclidean space which we will also introduce below.

Now we discuss several consequences. Unlike the case in number theory, where solutions to Diophantine equations are easier to calculate in modulo arithmetic, understanding area-minimizing representatives in mod $n$ homology is generally harder. 
	
Let us start with the problem of proving area-minimization. The method of calibrations  \cite{HL,HFrf} remains the  only known general method of proving area-minimization on compact manifolds and it applies to integral homology but not mod $n$ homology. In mod $n$ homology, there is no known general way of proving area-minimizing. For instance, to the author's knowledge, it is unknown if any degree $d$ algebraic subvariety of the standard complex projective space is area-minimizing mod $n$ or not for $d\ge 2.$ As a direct corollary of our main theorem, we deduce that
\begin{cor}\label{corn}
	Every area-minimizing integral current on a compact Riemannian manifold is area-minimizing mod $n$ for infinitely many $n$.
\end{cor}
Thus, all the classical calibrated  subvarieties on compact manifolds with special holonomy, including K\"{a}hler subvarieties, special Lagrangians \cite{RBcs,RBes}, associative/coassociative \cite{RBcs,AHNP} submanifolds are always area-minimizing mod $n$ for infinitely many $n.$

In the setting of proving area-minimization in Euclidean spaces, again calibrations is the dominant method. Except for directed slicing proposed by Lawlor \cite{GLds,GLsac}, essentially the only way to prove area-minimization mod $n$ is via area-non-increasing projections, including Frank Morgan's pioneering work \cite{FMcalv} (generalized in \cite[Lemma 2.11]{ZLa}) and the groundbreaking work of Lawlor \cite{GL}. Morgan \cite{FMcalv} successfully proved that complex algebraic subvarieties of degree $d$ in Euclidean space are area-minimizing mod $n$ for $n\ge 2d.$  Unfortunately, they all have limitations. For example, it is unknown \cite[Proposition 5.6]{FMcalv} if the examples \cite{HL} of special Lagrangian \cite{MHsc,MHsc2,DJsl,YZsmp,RBssl} cones, associative/coassociative \cite{RBas} cones and Cayley \cite{JLa,JL2r} cones are area-minimizing mod $n$ for some $n$ or not. The Euclidean version of our main theorem says that
\begin{rst}\label{rstc}
	An area-minimizing integral cone with subanalytic link is area-minimizing mod $n$ for $n$ sufficiently large.
\end{rst}		
The more precise statement of Result \ref{rstc} is Theorem \ref{thmcm} in Section \ref{secmtb}. Here by subanalytic link we mean that the intersection of the support of the cone with the unit sphere is a subanalytic set. To the author's knowledge, every area-minimizing integral cone known so far has subanalytic link. Concrete examples include every $3$-dimensional area-minimizing cone \cite{DSS1,DSS2,DSS3} and every cylindrical area-minimizing cone that is the product of a cone with smooth link with Euclidean space. Thus, Result \ref{rstc}  significantly enlarge the class of known mod $n$ area-minimizing cones.

The simplest type of singular area-minimizing cone is the so called pairs of planes, i.e., the union of two subspaces of the same dimension in Euclidean space. Area-minimizing pairs of planes among integral currents are classified in \cite{GLac,DN}, answering affirmatively the angle conjecture of Morgan \cite{FMtd}. To the author's knowledge, for general $n,$ only partial results were known by \cite{FMcalv,GL}. As a corollary of our method, we also solve the problem of classifying area-minimizing pairs of planes mod $n$ for $n$ at least $4.$
\begin{rst}\label{rstcp}
	A pair of oriented planes is area-minimizing mod $n$ for $n\ge 4$ if and only if the pair is an area-minimizing integral current. 
\end{rst}
The more precise statement of Result \ref{rstcp} is Theorem \ref{thmcpm} in Section \ref{secmtb}. For $n=2$ and $n=3$, we recall Morgan's conjectures \cite[Problem 16]{FMwtc} in Section \ref{secremp}.

Next, let us discuss the implication of our results on the regularity theory of area-minimizing currents. The state of the art on area-minimizing integral currents \cite{DMS1,DMS2,DMS3,BK1,BK2,BK3} prove the codimension $2$ (with respect to dimension of the current) countable rectifiability of singular set, while for area-minimizing currents mod $n$, \cite{DMSmodq,DHMSS,MWmodp} proved the codimension $1$ (with respect to dimension of the current) countable rectifiability of singular sets. In general, this $1$-dimensional regularity gap cannot be bridged due to concrete examples \cite[second remark after Theorem E]{MWmodp}. Corollary \ref{corn} tells us area-minimizing currents mod $n$ can have the same regularity as area-minimizing integral currents, e.g., enjoying better regularity, existence of fractal singular sets \cite{ZLa} and non-smoothable singular sets \cite{ZLns}. 

In the setting of Plateau problem on Euclidean spaces, Almgren asked \cite[Problem 4.3]{GMT}
\begin{quote}To what extent is it possible, for example, to correlate combinatorial complexity or singularity structure in area-minimizing currents with boundary properties?
\end{quote}
Inspired by \cite{FMmmn,EWW} in the hypersurface case, we prove that
\begin{cor}\label{thmp}
With smooth or subanalytic boundary, the solution to Plateau problem mod $n$ is obtained by integral currents for $n$ large enough, thus enjoying an improvement of regularity.
\end{cor}
See also Corollary \ref{coras} in Section \ref{secmtc}. We will later show that our main theorem is quantitatively sharp (Theorem \ref{thmvn}). The method we used to prove the sharpness of Result \ref{rstm} incidentally solved several other problems. In all previous discussions, calibrations remain the only known general method to prove area-minimizing. Thus, it is natural to ask if being calibrated is a generic intrinsic property of area-minimizing submanifolds \cite{MOMF}. In the positive direction, Federer proved in \cite[Section 5.10]{HFrf}
\begin{quote}
	Area-minimizing representatives of integral hypersurfaces classes on oriented Riemannian manifolds are calibrated by weakly closed measurable forms.	
\end{quote}
Negative evidence has been provided by \cite{HF,FMmc,BWmc,BKsf,KSsf,MKsf,MFsf} As a by-product of our method, we show that the generically calibrated is a false statement in higher codimensions:
\begin{rst}\label{rstnc}
On orientable manifolds, area-minimizing integral currents of codimension at least two are not generically calibrated.\end{rst}
Here by not generically calibrated, we mean that for each nonzero integral homology class, we can find non-empty open sets of metrics, in which, no area-minimizing integral current of the class can be calibrated by a weakly closed measurable form. Thus, we mean that the statement of calibrated in generic metrics is false, i.e., negation of generic metric. This choice of language is unfortunate but simpler for the introduction. We will introduce the detailed version Theorem \ref{thmnc} in Section \ref{secmtc}.

	Frank Morgan asked \cite[Problem 3.7]{GMT},
\begin{quote}
	Is the Cartesian product of two area-minimizing surfaces area-minimizing?
\end{quote}
We give the following answer.\begin{rst}\label{rstnp}
On a product of two closed oriented manifolds with non-vanishing Betti numbers, products of  area-minimizing representatives of some homology classes are not generically area-minimizing.
\end{rst}
Again, we mean the negation of the generic statement, i.e., existence of non-empty open sets of metrics in which products of area-minimizing integral currents in some homology classes are not area-minimizing. The more precise statement is Theorem \ref{thmkun} in Section \ref{secprod}.
\subsection{Sketch of Proof}\label{secsof}
Let us first give a sketch of proof of Result \ref{rstm} and related theorems. The proof is essentially inspired by \cite{FMmmn,EWW}. By the universal coefficient, every integral homology class $[\Si]$ on a compact manifold has a canonical image $[\Si]^{\Z/n\Z},$ the mod $n$ reduction of $[\Si]$, in $\Z/n\Z$ homology. Roughly speaking, to prove Result \ref{rstm}, we prove directly that for infinitely many $n$, the mod $n$ area-minimizing representative of $[\Si]^{\Z/n\Z}$ lifts to an integral area-minimizing representative of $[\Si]$. Similarly, in the setting of Plateau problems in Euclidean spaces, heuristically we want to lift $\Z/n\Z$-area-minimizing representatives of relative homology classes. The lifting consists of three steps:
\begin{enumerate}
\item Using monotonicity formulas, we get a priori density upper bounds on the mod $n$ area-minimizing representatives of $[\Si]^{\Z/n\Z}$,
\item Using the density upper bounds and regularity theory in \cite{DHMS}, we kill $Y$-shaped singularities and lift the mod $n$ area-minimizing representatives to area-minimizing integral cycles,
\item Using Federer's theory of real flat chains \cite{HFrf} as asymptotic data, we determine the homology class of lifted area-minimizing integral cycles and finish the proof.
\end{enumerate}
The first step is carried out in Section \ref{secmon} and Section \ref{secmonb}. 

The second step is carried out in Section \ref{secreg}. Roughly speaking, the obstruction to lift a mod $n$ cycle into an integral cycles lies in the $(d-1)$-dimensional stratum of the mod $n$ cycle (Lemma \ref{lemdv}). If a mod $n$ area-minimizing cycle does not have $Y$-shaped triple junction like singularities, then it is an integral cycle by the regularity theory developed in \cite{DHMS}. 

To kill the $Y$-shaped singularities of  the mod $n$ area-minimizing cycles representing $[\Si]^{\Z/n\Z}$, we observe that, for large $n$, the classification of $Y$-shaped singularities gives a priori density lower bounds of $Y$-shaped singularities growing linearly on $n$. Then the the a priori density bound from the first step gives a contradiction

We achieve the last step in Section \ref{secfed}. The goal is to tell apart different integral homology classes that has the same $\Z/n\Z$-reduction. This requires investigating in detail the growth of the mass of area-minimizing representatives in integral homology, which we will prove to be ill-behaved on finite subsets (Corollary \ref{corbad}). Fortunately, Federer \cite{HFrf} developed a very general theory, showing that the mass of area-minimizing representatives in real homology controls the asymptotic growth of those in integral homology. which allows us to finish the proof of Result \ref{rstm}. Corollary \ref{corn} follows directly. The proof of Result  \ref{rstc}, Result  \ref{rstcp} and Corollary \ref{thmp} in the Euclidean setting is the same by replacing homology in our argument with relative homology. 

Result \ref{rstnc}, Result \ref{rstnp} uses a different perspective that arises from Theorem \ref{thmvn} in Section \ref{secmtc} in the next section, when investigating the sharpness of Result \ref{rstm}. Roughly speaking, all three rely on constructing metrics in which area-minimizing representatives of $[\Si]$ have very large area compared to those of some multiple of $[\Si]$. The construction is as follows. Suppose $[\Si]$ is not a torsion class, i.e., $[\Si]$ generating a free abelian group, and $\mu_{[\Si]}[\Si]$ can be represented by a connected embedded submanifold $N.$ Take a tubular neighborhood $U(N)$ of $N$. Then algebraic topology tells us that any cycle contained in $U(N)$ must represents a homology class that is a multiple of $\mu_{[\Si]}[\Si]$. Thus, any area-minimizing representative of $[\Si]$ cannot stay in $U(N).$ Zhang \cite{YZt,YZa} (summarized in Lemma \ref{lemzhang}) gives us a smooth metric in which $N$ is an area-minimizing representative of $\mu_{[\Si]}[\Si].$ Making the metric outside of $U(N)$ very large, we can force any area-minimizing representative of $[\Si]$ to have area much larger than $[\Si]$ since it cannot stay inside $U(N).$
\subsection{Organization}
In Section \ref{secmt} we will give more technical introductions and give more precise formulations of our main results introduced in Section \ref{hsintro}. The rest of the manuscript will be devoted to prove the theorems in Section \ref{secmt}.

In Section \ref{secbd} we will introduce the basic definitions and collect several technical preliminaries.

In Section \ref{secmon} and \ref{secmonb} we will prove a priori density estimates carrying out the first step outlined in Section \ref{secsof}.

In Section \ref{secreg}, we will use regularity theory developed in \cite{DHMS} to carry out the second step outlined in Section \ref{secsof}.

In Section \ref{secfed}, we will carry out the third step outlined in Section \ref{secsof} using Federer's seminal work \cite{HFrf}.

In Section \ref{secpm}, we will wrap up our argument and prove  Theorems \ref{thmv}, \ref{thmvg}, \ref{thmvb}, \ref{thmcpm} and  Corollary \ref{coras}..

In Section \ref{secd=1}, we will deal with the case of $d=1$ in Theorem \ref{thmcm} and provide an alternative argument that does not use Federer's work \cite{HF} for Plateau problems.

In Section \ref{secpfs} we will prove Theorems \ref{thmvn} and \ref{thmnc}.

In Section \ref{secpfp} we will prove Theorem \ref{thmkun}.

In Section \ref{secrem} we will give some remarks.
	\section*{Acknowledgements}
This manuscript is part of the author's Ph.D. dissertation and  supersedes \cite{ZLar,ZLnco}. The author would like to express sincere thanks towards the defense committee, Professors Camillo De Lellis, Leon Simon and Shing-tung Yau for their unwavering support. The author is immensely grateful for the teachings of his advisor, Professor De Lellis. The author is deeply indebted to Professors Simon and Yau for their extraordinary generosity, guidance, and kindness throughout the years. On the occasion of Professor Simon's 81st birthday and Professor Yau's 77th birthday, the author warmly extends best wishes to them.  Sincere thanks go to the outside dissertation readers, Professors Robert Bryant and Antoine Song, for many helpful suggestions.

Special tribute is paid to Professors William Allard and Frederick Almgren. Sincere thanks are paid to Professors Frank Morgan and Antoine Song, whose works are the inspirations for the author to go from local regularity to global geometry. Professor Morgan has also raised several suggestions that ultimately inspire a significant expansion of the manuscript. Many inspirations also come from Professor Brian White's foundational works on flat chains. A conversation with Professor Michael Freedman partially inspired Theorem \ref{thmvn}. Special thanks go to Professor Yang Li, who have raised many helpful suggestions. He suggested the author to write up Theorem \ref{thmkun}. Thanks goes to Professor Zhihan Wang for explaining the non-subanalytic nature of Hardt-Simon foliation. 
\section{Detailed introduction to the main theorems}\label{secmt}In this section, we will introduce the geometric measure theory of flat chains and write down a more precise formulation of Result \ref{rstm}. Our main theorems have two different set-ups, one on compact manifolds and the other on Euclidean spaces, so we will treat them differently. Result  \ref{thmp} is different in theme from other theorems and we will discuss the setting as well. The reader is highly suggested to read this section at least once in order to understand the basic ideas. More technical lemmas and backgrounds will be left to the next section. 
\subsection{Compact setting}\label{secmtc}Let us start with the compact manifold setting.
\begin{assump}\label{assumpm}(Compact setting)	Assume that 
	\begin{itemize}
		\item $M$ is a $(d+c)$-dimensional compact closed not necessarily orientable smooth manifold, with $d,c\in\Z_{\ge 1},$
		\item $[\Si]\in H_d(M,\Z)$ is a $d$-dimensional integral homology class,
		\item $g$ is a smooth Riemannian metric on $M.$
\end{itemize} \end{assump}Recall \cite{AH} that $H_d(M,\Z)$ is a finitely generated abelian group which can be written as a splitting up to choices of basis,
\begin{align}\label{decomph}
H_d(M,\Z)=\Z^{b_d}\oplus\bigoplus_{i\in I}\Z/p_i^{\nu_i}\Z,
\end{align}where $b_d$ is the $d$-th Betti number of $M$, $I$ is a finite set of indices, the $p_i$ are not necessarily distinct prime numbers and $\nu_i\in \Z_{\ge 1}$. By \cite[Theorem 1.57(c)]{JMgt}, the integers $\{b_d\}\cup_{i\in I}\{p_i,v_i\}$ and the subgroup $\bigoplus_{i\in I}\Z/p_i^{\nu_i}\Z$ are the canonically defined and independent of the choice of basis. \begin{defn}\label{defntm}
Use $G_d(M,\Z)$ to denote the torsion subgroup of $H_d(M,\Z)$, i.e., 
\begin{align*}
G_d(M,\Z)=\bigoplus_{i\in I}\Z/p_i^{\nu_i}\Z.
\end{align*}Define the $d$-th torsion number $\tau_d$ of $M$ to be
\begin{align*}
\tau_d=\operatorname{lcm}_{i\in I}p_i^{\nu_i},
\end{align*}where $\operatorname{lcm}$ means the 
least common multiple. Set $\tau_d=1$ if $H_d(M,\Z)$ is torsion-free. 
\end{defn}
Like the $d$-th Betti number $b_d$, the $d$-th torsion number $\tau_d$ is an invariant of $M.$\begin{defn}We classify the integral homology class $[\Si]$ according to its algebraic properties.\label{defnhom}
\begin{itemize}
\item 	We say $[\Si]$ is a torsion class if there is an integer $v$ such that $v[\Si]=0,$ i.e., belonging to $G_d(M,\Z)$. Otherwise we say $[\Si]$ is a non-torsion class,
\item If $[\Si]$ is a non-torsion class, $M$ is orientable, and $d\ge1 ,c\ge 2$, then the representable multiplicity $\mu_{[\Si]}$ of $[\Si]$ is defined to be the smallest integer $\mu_{[\Si]}\in\Z_{\ge 2}$ such that $\mu_{[\Si]}[\Si]$ can be represented by a smoothly embedded connected submanifold.
\end{itemize}
\end{defn}		
The well-definedness of $\mu_{[\Si]}$ is proved in Lemma \ref{lemthom}. 
By the universal coefficient theorem \cite[Theorem 3A.3]{AH}, for any coefficient ring $R$, we have a natural short exact sequence
\begin{align}\label{equct}
	0&\rightarrow H_d(M,\Z)\ts R\overset{\iota}{\rightarrow} H_d(M,R)\rightarrow\operatorname{Tor}(H_{d-1}(M,\Z),R)\to 0.
\end{align} 
\begin{defn}\label{defni}
	The $R$-reduction $[\Si]^R$ of the integral homology class $[\Si]$ is defined to be
	\begin{align*}
		[\Si]^R=\iota([\Si]\ts 1)\in H_d(M,R),
	\end{align*}where $\iota$ is the natural injection in (\ref{equct}) and $1$ is the unit element in $R.$ \end{defn}

	We need a norm $\no{\cdot}^R$ on $R$ in order to measure the area of chains in $R$ homology. 
	\begin{defn}\label{defnnm}
		We say a map $\no{\cdot }^R:R\to \R_{\ge0}$ is a norm on $R$ if $\no{\cdot}^R$ satisfies the following three properties:
		\begin{itemize}
			\item (reversibility) for all $e\in R$, we have
			\begin{align*}
				\no{-e}^R=\no{e}^R,
			\end{align*}
			\item (triangle inequality) for all $e,f\in R$, we have
			\begin{align*}
				\no{e+f}^R\le \no{e}^R+\no{f}^R.
			\end{align*}
			\item (positivity) we have the equivalence of propositions, for all $e\in R,$
			\begin{align*}
				\no{e}^R=0\eq e=0. 
			\end{align*} 
		\end{itemize}
	\end{defn}We want to emphasize that we do not require the usual homogeneity axiom, i.e., $\no{ve}^R=|v|\no{e}^R$ for $e\in R,v\in\Z$ on $R,$ but only use the weaker reversibility axiom. This is an intentional omission in order to measure elements in groups with torsions. 
	
	Recall \cite{WFfc,BWrc,BWdt} that $\no{\cdot}^R$ induces a natural flat distance $\mathcal{F}^R$ and mass norm $\ms^R_g$ on the space of $R$-coefficient polyhedron chains on the Riemannian manifold $(M,g)$. 
	\begin{defn}\label{defnfc}
		The $R$-flat chain $\mathcal{F}_g(M,R)$ in the metric $g$ is the $\mathcal{F}^R$ closure of the set of $R$-coefficient polyhedron chains.
	\end{defn} Intuitively speaking the mass $\ms_g^R$ of a $d$-dimensional chain is just the $d$-dimensional area of a chain counted with multiplicity by $\no{\cdot}^R.$ For instance, by Definition \ref{defnnorm}, $$\ms^{\Z/4\Z}(\pm 2 S^1)=\no{\pm 2}^{\Z/4\Z}\ms^{\Z/4\Z}(S^1)=2\cdot 2\pi.$$
	
	By \cite[Section 4.2.16]{HF}, the finite area $\Z$-coefficient flat chains representing $[\Si]\in H_d(M,\Z)$ are precisely the classical integral currents representing $[\Si]$. As we are concerned with area-minimization, we will use the literature convention of using integral currents to represent integral homology.
	\begin{defn}When the canonical homomorphism $\Z\to R$ is norm non-increasing, define the $R$-reduction $T^R$ of an integral current $T\in[\Si]$ to be
		\begin{align*}
			T^R=\iota(T\ts 1)\in [\Si]^R.
		\end{align*}
	\end{defn}
	Here $\iota$ is the canonical homomorphism in (\ref{equct}), which also exists on the chain level. It is straightforward to verify that $T^R$ is well-defined.
	
	We want to emphasize that the set of $R$-flat chains in $[\Si]^R$ is much larger than the $R$-reduction of integral currents in $[\Si]$. For instance, in $S^4$, where the second homology is trivial, an embedded $\mathbb{RP}^2$ is a mod $2$ cycle, but is not a $\Z/2\Z$-reduction of any integral cycle.
	
	Recall Definition \ref{defnfc}.
	\begin{defn}\label{defnrarea}
		Define the $R$-area $\no{[\Si]}^R_g$ of $[\Si]^R$ to be
		\begin{align}\label{defnra}
			\no{[\Si]}^R_g=\inf_{\substack{T\in\mathcal{F}_g(M,R)\\ T\in[\Si]^R}}\ms_g^R(T).
		\end{align}
		We say an $R$-flat chain $T\in[\Si]^R$ is $R$-area-minimizing in $[\Si]^R$ with respect to metric $g,$ if $T$ achieves the $R$-area of $[\Si]^R$ defined in (\ref{defnra}), i.e.,
		\begin{align*}
			\ms^R_g(T)=\no{[\Si]}^R_g.
		\end{align*} 
	\end{defn}
	Intuitively speaking, $\no{[\Si]}^R_g$ is the infimum of area of all $R$-coefficient polyhedron chains representing the $R$-reduction $[\Si]^R$ of $[\Si]$. By the discussion before Definition \ref{defnrarea}, we have
	\begin{fact}
		The following inequality holds
		\begin{align}\label{eqnormcomp}
			\no{[\Si]}^R_g\le \inf_{T\in[\Si]}\ms_g^{R}(T^R).
		\end{align}
	\end{fact}
	In this manuscript, we will always use the following classical norm conventions:
	\begin{defn}\label{defnnorm}Define
		\begin{align*}
			\no{\cdot}^{\Z}=&|\cdot|,\\
			\no{\cdot}^{\R}=&|\cdot|,\\
			\no{\cdot}^{\Z/n\Z}=&|\cdot|_{(-\frac{n}{2},\frac{n}{2}]}.
		\end{align*}
	\end{defn}
	Here $|\cdot|$ is the absolute value on $\R$ and $|\cdot|_{(-\frac{n}{2},\frac{n}{2}]}$ means the absolute value of the unique lift of an element in $\Z/n\Z$ to $(-\frac{n}{2},\frac{n}{2}]$. With $\no{\cdot}^R$ defined as above, when $R=\Z,\R$ or $\Z/n\Z$ with $n\in\Z_{\ge 2},$ there always exists at least one area-minimizing $R$-flat chain representative of $[\Si]^R$ \cite{HFrf,WFfc,HF}. 
	
	The following assumption will always be assumed.
	\begin{assump}\label{assumpcoef}
		Set $R=\R$ or $\Z$ or $\Z/n\Z$ for $n\in\Z_{\ge 2}$ and use Definition \ref{defnnorm} for $\no{\cdot}^R.$
	\end{assump}
	It is straightforward to verify that by Definition \ref{defnnorm}, the canonical ring homomorphism $\Z\to R$ is norm non-increasing. Thus, we always have
	\begin{align*}
		\ms^R_g(T^R)\le \ms_g^{\Z}(T).
	\end{align*}
	By (\ref{eqnormcomp}), we deduce that
	\begin{fact}\label{fctncmp}The following inequality always holds
		\begin{align}\label{eqinnorm}
			\no{[\Si]}^R_g\le \no{[\Si]}^{\Z}_g.
		\end{align}
	\end{fact}
	We emphasize again that $R$-reduction of a homology class introduces many more representatives, e.g., unorientable ones for $R=\Z/2\Z$, $Y$-shaped union of three rays starting at the origin for $R=\Z/3\Z,$ the operator of $L^2$ inner product with coclosed differential forms for $R=\R,$ etc. Thus, it is not expected and in fact not true in general that equality in (\ref{eqinnorm}) can hold. For this purpose, it is beneficial to give the following definition
	\begin{defn}
For a subset $S\s H_d(M,\Z),$		define the set of $R$-area-minimizing representatives $\mn_g^R\left(S\right)$ of $S^R$ to be
		\begin{align*}
			\mn_g^R\left(S\right)=\{T|T\in \mathcal{F}_g(M,R),[T]\in S^R,\ms^R_g(T)=\inf_{[\Gamma]\in S^R}\no{[\Gamma]}^R_g\}.
		\end{align*} 
	\end{defn}When (\ref{eqinnorm}) is strict inequality, we may have
	\begin{align}\label{eqesint}
		\mn_g^R\left([\Si]\right)\cap \left(\mn_g^{\Z}\left([\Si]\right)\right)^R=\es.
	\end{align} 
 Recall Assumption \ref{assumpm}, Definition \ref{defntm} and Definition \ref{defni}. We will be ready to state the more detailed version of Result \ref{rstm}.\begin{thm}\label{thmv}
		For all $[\Si]\in H_d(M,\Z)$, there exists a positive integer $N_{[\Si],g}$, and an open set $\Om_{[\Si],g}$ containing $g$ in the space of Riemannian metrics, both depending only on $[\Si]$ and $g$, such that for all metrics $h\in\Om_{[\Si],g},$ all $n\ge N_{[\Si],g},$ with $\tau_d|n$    
		\begin{itemize}
			\item equality holds in (\ref{eqinnorm}), i.e., $			\no{[\Si]}_h^{\Z/n\Z}=\no{[\Si]}_h^\Z,$
			\item the $\Z/n\Z$-reduction restricted to $\mn_h^{\Z}\left([\Si]\right)$ is a one to one map from  		$\mn_h^{\Z}\left([\Si]\right)$ to $\mn_h^{\Z/n\Z}\left([\Si]\right)$.	
		\end{itemize} 
	\end{thm}
	Here the inverse to $\Z/n\Z$-reduction restricted to $\mn_g^{\Z}\left([\Si]\right)$-reduction equals Federer's ``representative" modulo $n$ \cite[Section 4.2.26]{HF}. Unfortunately, Federer used the word ``representative" in a totally different way than we did, and the reader should not confuse Federer's ``representative" with our use of representatives of homology classes. The reader might wonder what happens when $\tau_d\not|n.$ 
	Using the same notations as in Theorem \ref{thmv}, we have the following 	\begin{thm}\label{thmvg} For all metrics $h\in\Om_{[\Si],g},$ all $n\ge N_{[\Si],g},$ with $\tau_d\not|n$, we have
		\begin{align*}
		\left(\mn_h^{\Z}\big([\Si]+nG_d(M,\Z)\big)
		\right)^{\Z/n\Z}=\mn_h^{\Z/n\Z}([\Si]),
		\end{align*} and the $\Z/n\Z$-reduction is a bijection from $\mn_h^{\Z}\big([\Si]+nG_d(M,\Z)\big)$ to $\mn_h^{\Z/n\Z}([\Si])$.
	\end{thm}
	Note that  by \cite[Section 3.2 p. 215 property (5)]{AH}, $nG_d(M,\Z)$ is the kernel of $\Z/n\Z$-reduction when restricted to $G_d(M,\Z)$
	Theorem \ref{thmvg} means we can no longer recover $\mn_h^{\Z}([\Si])$ when $\tau_d\not|n.$   
	For example, when $n$ is a large prime number, then $nG_d(M,\Z)=G_d(M,\Z)$, and if some element of $[\Si]+G_d(M,\Z)$ have smaller $\Z$-area than $\no{[\Si]}_h^{\Z}$ then we can never recover  $\mn_h^{\Z}([\Si])$ from $\no{[\Si]}_h^{\Z/n\Z}$ for such $n.$
	
	Though the statement of Theorem \ref{thmv} and Theorem \ref{thmvg} only describes information from $\Z/n\Z$ homologies, like the proof of the Hasse-Minkowski theorem, information obtained from both real homology and $\Z/n\Z$ homologies is indispensable to the proof of Theorem \ref{thmv} . 
	
	State of the art regularity proved in \cite{LScy,DMS1,DMS2,DMS3,DMSmodq} by De Lellis and his collaborators  shows that the singular set of $T$ is of Hausdorff dimension
	\begin{align*}
		\begin{cases}
			(d-2)-\textnormal{rectifiable, if }T\textnormal{ is area-minimizing in integral homology},\\
			(d-1)-\textnormal{rectifiable, if }T\textnormal{ is area-minimizing in }\Z/n\Z \textnormal{ homology}.
		\end{cases}
	\end{align*}See also \cite{BK1,BK2,BK3} for an independent proof of the first case.
	There is a $1$-dimension gap between these two results. In the case of $c=1,$ the singular sets of area-minimizing representatives of integral homology classes have Hausdorff dimension at most $(d-7)$ \cite{NV}, and the gap is increased to $6$ dimensions in general. The dimension gap in regularity cannot be overcome, as the regularity results above are sharp. However, as a corollary of Theorem \ref{thmv}, we can prove that
	\begin{cor}\label{coras}
		Area-minimizing representatives of $[\Si]^{\Z/n\Z}$ are $\Z/n\Z$-reductions of area-minimizing integral currents for $n$ large enough, thus enjoying much better regularity than general $\Z/n\Z$ area-minimizing representatives.
	\end{cor}
This  generalize Frank Morgan's result \cite{FMmmn} into the compact setting. 

	One can ask what information about area-minimizing representatives in integral homology is lost in $\R$ and $\Z/n\Z$ homology for any finite sets of $n$. To this end, we have the following theorem, which shows the sharpness of Theorem \ref{thmv} and the impossibility of finding a uniform upper bound for $N_{[\Si],g}$ in Theorem \ref{thmv}. Thus, to recover the $\Z$-area-minimizing representative of $[\Si]$, we do need information from real homology and all $\Z/n\Z$ homologies, much like the role of real and all mod prime power solutions in the Hasse-Minkowski theorem. Recall Definition \ref{defnhom}.
	\begin{thm}\label{thmvn}
		If $d\ge 1,c\ge 2,$ $M$ is orientable, and $[\Si]$ a non-torsion class, then for any real number $\rh>1$ and all integers $n\ge 2$ with $\operatorname{gcd}\left(n,\mu_{[\Si]}\right)=1,$  there exists a non-empty open set $\Om'_{[\Si],n,\rh}$ in the space of Riemannian metrics, depending on $[\Si],n,$ and $\rh,$ such that for all metrics $h\in\Om'_{[\Si],n,\rh}$ we have
		\begin{align}
			\label{eqzv}	{\no{[\Si]}^{\Z}_h}>& \rh{\no{[\Si]}^{\Z/n\Z}_h},\\
			\label{eqzr}	{\no{[\Si]}^{\Z}_h}>&\rh{\no{[\Si]}^{\R}_h},\\
			\label{eqzz}	{\no{[\Si]}^{\Z}_h}>& \rh{\no{\mu_{[\Si]}[\Si]}^{\Z}_h}.
		\end{align}
	\end{thm}
	We want to remark that if $\operatorname{gcd}(\tau_d,\mu_{[\Si]})=1,$ e.g., $\tau_d=1,$ then we can always find  $\operatorname{gcd}\left(n,\mu_{[\Si]}\right)=1$ and $\tau_d|n,$ so that Theorem \ref{thmvn} genuinely shows the sharpness of Theorem \ref{thmv}. 
	
In the setting of Plateau problems in Euclidean space, Frank Morgan \cite{FMmc}, Brian White \cite[Problem 1.13]{GMT} have all raised the following question
\begin{quote}
	Is it possible to fill in a multiple of a boundary with arbitrarily small area compared to the area of the original filling of the boundary?
\end{quote}
Robert Young \cite{RY} showed that for any fixed multiple, this is impossible. By (\ref{eqzz}), we show that the opposite is true in the setting of area-minimizing representatives on compact oriented manifolds of codimension at least $2$.

An extension of (\ref{eqzr}) gives the following theorem\begin{thm}\label{thmnc}
If $c\ge 2$ and $M$ is orientable, then for any $[\Si]\not=0,$ there exists a non-empty open set $\Omega_{[\Si],\R}$ in the space of Riemannian metrics, such that for any metric $h\in \Omega_{[\Si],\R}$ and any $T\in\mn_h^\Z([\Si])$, there is no weakly closed measurable $d$-form calibrating $T$.\end{thm}
	\subsection{Euclidean setting}\label{secmtb}The Plateau problem in Euclidean space is about finding the existence of $d$-dimensional $R$-flat chains of minimal possible $R$-mass filling a given boundary supported in a set $W$. The most natural setting for solving Plateau problems is finding $R$-area minimizing representatives of relative homology classes in $H_d(\R^{d+c},W,R)$. Thus, we propose the following assumption. 
\begin{assump}\label{assumpmb}	Assume that $c,d\in\Z_{\ge1}$ and
	\begin{enumerate}
	\item $W$ is either a $(d-1)$-dimensional compact smooth submanifold or a   compact $(d-1)$-dimensional subanalytic set of $M=\R^{d+c}$,\label{assumpmbr}
	\item $[\Si]\in H_d(\R^{d+c},W,\Z)$ is a $d$-dimensional relative integral homology class with $d\ge 2$.
	\end{enumerate} \end{assump}
Here $W$ is allowed to have multiple connected components. All discussions in this subsection is based on Assumption \ref{assumpmb}.

The theory of $R$-flat chains discussed in Section \ref{secmtc} can be developed in Euclidean space as well. If the $W$ is  well-behaved, i.e., Lipschitz neighborhood retract \cite[Sections 4.4.1, 4.4.5, 4.4.6]{HF}\cite[Section 3.4]{HFrf}, then the homology of $R$-flat chains with boundaries supported in $W$ is canonically isomorphic to relative homology $H_d(\R^{d+c},W,R)$ and the above interpretation can be realized literally. 

Unfortunately, the relative set $W$ we need to deal with in general are not Lipschitz neighborhood retracts, e.g., branched minimal surfaces in unit spheres. Our way to salvage this technicality is to reformulate the setting using the long exact sequence. Recall the long exact sequence for relative homology \cite[Theorem 2.16]{AH}:
\begin{align*}
	\cdots\to H_d(\R^{d+c},R)\to H_d(\R^{d+c},W,R)\xrightarrow{\pd} H_{d-1}(W,R)\to H_{d-1}(\R^{d+c},R)\to\cdots,
\end{align*}where $\pd$ is the boundary map of $R$-chains.
When $d\ge 2,$ we deduce that
\begin{align}\label{eqisohom}
	H_d(\R^{d+c},W,R)\overset{\pd}{\cong} H_{d-1}(W,R).
\end{align}
For subanalytic $W$, De Pauw and Hardt  \cite[Section 5.3]{DHlii} proved that the $(d-1)$-dimensional homology of $R$-flat chains supported in $W$ are canonically isomorphic to the simplicial homology of $W$ in $R$ for $R=\Z$ and $\Z/n\Z$ with $n\in\Z_{\ge2}$. For $R=\R$, we will prove in Lemma \ref{fcthomr} that the same holds.

Thus, it makes sense to speak to speak of $(d-1)$-dimensional $R$-flat chains representing homology classes in $H_{d-1}(W,R),$ which by (\ref{eqisohom}) gives us a way to discuss $R$-flat chains representing relative homology classes in $H_d(\R^{d+c},W,R)$. We use the following convention
\begin{defn}\label{defntsw}
Assume $d\ge 2$. For $d$-dimensional $R$-flat chains $T$ with $\pd T$ supported in $W,$ we define
\begin{align*}
	[T]=\pd\m[\pd T]\in H_d(\R^{d+c},W,R). 
\end{align*}
\end{defn}
Let us make two remarks about our definition. First, an $R$-flat chain $T$ can represent a relative homology class in $H_d(\R^{d+c},W,R)$ if and only if $\pd T$ represents a homology class in $H_{d-1}(W,R).$ Second for two $R$-flat chains by definition we have
\begin{align*}
	[T]=[S]\textnormal{ if and only if }[\pd T]=[\pd S].
\end{align*}  Thus, now we can legitimately speak of $R$-flat chains representing relative homology classes of $H_d(\R^{d+c},W,R)$ when $d\ge 2.$

Introduce the following definition.
	\begin{defn}\label{defnrareab}
	For $d\ge 2,c\ge1$ define the $R$-area $\no{[\Si]}^R_\de$ of $[\Si]^R$ to be
	\begin{align}\label{defnrab}
		\no{[\Si]}^R_\de=\inf_{\substack{T\in\mathcal{F}_\de(\R^{d+c},R)\\ T\in[\Si]^R}}\ms_\de^R(T).
	\end{align}
	We say an $R$-flat chain $T\in[\Si]^R$ is $R$-area-minimizing in $[\Si]^R$ with respect to metric $g,$ if $T$ achieves the $R$-area of $[\Si]^R$ defined in (\ref{defnra}), i.e.,
	\begin{align*}
		\ms^R_\de(T)=\no{[\Si]}^R_\de.
	\end{align*} 
\end{defn} 
Here $[\Si]^R$ is the $R$-reduction map $i([\Si]\ts R)$ in the universal coefficient theorem in relative homology \cite[Corollary 3A.4]{AH}:
\begin{align}
	0&\rightarrow H_d(\R^{d+c},W,\Z)\ts R\overset{\iota}{\rightarrow} H_d(\R^{d+c},W,R)\rightarrow\operatorname{Tor}(H_{d-1}(\R^{d+c},W,\Z),R)\to 0.
\end{align}  and $\de$ is the standard flat metric on $\R^{d+c}$. We still have \begin{fact}\label{fctncmpb}The following inequality always holds
\begin{align}\label{eqinnorm}
	\no{[\Si]}^R_g\le \no{[\Si]}^{\Z}_g.
\end{align}\end{fact}
Again, we define
	\begin{defn}\label{defnpf}
	For a subset $S\s H_d(\R^{d+c},W,\Z),$		define the set of $R$-area-minimizing representatives $\mn_\de^R\left(S\right)$ of $S^R$ to be
	\begin{align*}
		\mn_\de^R\left(S\right)=\{T|T\in \mathcal{F}_\de(\R^{d+c},R),[T]\in S^R,\ms^R_\de(T)=\inf_{[\Gamma]\in S^R}\no{[\Gamma]}^R_\de\}.
	\end{align*} 
\end{defn}Our main theorem in the Euclidean setting is as follows
\begin{thm}\label{thmvb}
	For all $[\Si]\in H_d(\R^{d+c},W,\Z)$, there exists a positive integer $N_{[\Si],\de}$, depending only on $[\Si]$, such that for all $n\ge N_{[\Si],\de},$
	\begin{itemize}
		\item we have $			\no{[\Si]}_\de^{\Z/n\Z}=\no{[\Si]}_\de^\Z,$
		\item the $\Z/n\Z$-reduction restricted to $\mn_\de^{\Z}\left([\Si]\right)$ is a one to one map from  		$\mn_h^{\Z}\left([\Si]\right)$ to $\mn_\de^{\Z/n\Z}\left([\Si]\right)$.	
	\end{itemize} 
\end{thm}
We will prove in Lemma \ref{fcthomr}  that $H_d(\R^{d+c},W,\Z)$ is a free abelian group. Thus, \ref{thmvb} is a genuine analogue of Theorem \ref{thmv} with $\tau_d=1.$

Here the inverse to $\Z/n\Z$-reduction restricted to $\mn_\de^{\Z}\left([\Si]\right)$-reduction again equals Federer's ``representative" modulo $n$ \cite[Section 4.2.26]{HF}. The precise statement for Corollary  \ref{rstc} and Result  \ref{thmp} is as follows
\begin{thm}\label{thmcm}
	Assume that $d\ge 1,c\ge 1$, and $B$ is an integral current cycle, i.e., $\pd B=0,$ supported in $W,$ then exists $N_{B}\ge 0$ so that for all integers $n\ge N_B,$ the $\Z/n\Z$-reduction of an $\Z$-area-minimizing integral current $T$ with boundary $B$  is $\Z/n\Z$-area-minimizing.
\end{thm}
Theorem \ref{thmcm} is a direct corollary of Theorem \ref{thmvb} when $d\ge 2$. A separate proof will be given in Section \ref{secd=1} for the case of $d=1$, as $H_1(\R^{d+c},W,R)\not\cong H_0(W,R)$ in general. 

The precise statement of Result  \ref{rstcp} is as follows. 
\begin{thm}\label{thmcpm}
Assume that $P,Q$ are oriented $d$-dimensional subspaces of $\R^{d+c}$, with $d\ge 2,c\ge 1,n\ge4.$ Then  $P^{\Z/n\Z}+Q^{\Z/n\Z}$ is $\Z/n\Z$ area-minimizing if and only if $P+Q$ is $\Z$-area-minimizing.
\end{thm}
It turns out Theorem \ref{thmcpm} reduces to a more specific case as follows. We need a refinement of Assumption \ref{assumpmb}.
\begin{assump}\label{assumpplane}
	For $c=d\ge 2,$ assume that $P,Q$ are two oriented $d$-dimensional planes intersecting only at the origin in $\R^{d+c}$. 
\end{assump}
We will show later that Theorem  \ref{thmcpm} follows from the following lemma
\begin{lem}\label{lemcpm}
	For $n\ge 4$, the $\Z/n\Z$-flat chain $P^{\Z/n\Z}+Q^{\Z/n\Z}$ is $\Z/n\Z$-area-minimizing if and only if $P+Q$ is $\Z$-area-minimizing.
\end{lem}
	\subsection{Product setting}\label{secprod}
We propose an alternative interpretation of our main theorems:
	\begin{quote}
		\emph{We are quantifying the effect of $R$-reduction in the universal coefficient theorem on area-minimizing representatives.}
	\end{quote} Theorem \ref{thmv} says that for integers $n$ large enough and divisible by $\tau_d$, the $\Z/n\Z$-reduction sends $\Z$-area-minimizing representatives bijectively to $\Z/n\Z$-area-minimizing representatives, while Theorem \ref{thmvn} shows that in non-empty open sets of metrics $R$-reduction fails drastically at sending $\Z$-area-minimizing representatives to $R$-area-minimizing representatives for $R=\R$ and $R=\Z/n\Z$ with $n\in\Z_{\ge 2}.$ 
	
	Thus the following problem provides a natural framework for our results.
	\begin{prob}
		For any natural homomorphism among homology groups, quantify the effect of the homomorphism on area-minimizing representatives.
	\end{prob}
	To give another illustration of the above problem beyond our setting of the universal coefficient theorem, we will quantify the effect of K\"{u}nneth formula on the homology of products. First, let us introduce another manifold.
	\begin{assump}\label{assumpmp}
		Assumption \ref{assumpm} holds with another set of objects $M',d',c',[\Si'],g'$ in place of $M,d,c,[\Si],g$.
	\end{assump}
	By the the K\"{u}nneth formula \cite[Theorem 3B.6]{AH}, for every natural number $l\in \N,$ we have,
	\begin{align}\label{eqkf}
		0\to \bigoplus_{i\in\N}H_i(M,\Z)\ts H_{l-i}(M',\Z)\overset{\mathfrak{i}}{\rightarrow} H_l(M\times M',\Z)\to\bigoplus_{i\in\N}\operatorname{Tor}\left(H_i(M,\Z),H_{l-i-1}(M',\Z)\right)\to0,
	\end{align}
	and when $R$ is a field \cite[Corollary 3B.7]{AH}, e.g., $R=\R$ or $R=\Z/p\Z$ with $p$ a prime number,
	\begin{align}\label{eqkff}
		0\to \bigoplus_{i\in\N}H_i(M,R)\ts H_{l-i}(M',R)\overset{\mathfrak{i}}{\rightarrow} H_l(M\times M',R)\to0.
	\end{align}
	\begin{defn}
		Define
		\begin{align}\label{defncrsp}
			[\Si]\times[\Si']=\mathfrak{i}\left([\Si]\ts [\Si']\right).
		\end{align}
	\end{defn}
	On the chain level, the tensor product on the left hand side of (\ref{eqkf}) and (\ref{eqkff}) are induced by cross products of chains \cite[Section 3.B]{AH}. Thus, the product defined in (\ref{defncrsp}) corresponds to the Cartesian product of $R$-flat chains for $R=\Z,\R$ or $\Z/p\Z$ with $p$ a prime number. 
	
	As there are many representatives of $[\Si]\times[\Si']$ which do not come from Cartesian products of cycles, we have, \begin{fact}
		The following inequality holds,
		\begin{align}\label{ineqprod}
			\no{[\Si]}^R_{g}\times \no{[\Si']}^R_{g'}\ge 
			\no{[\Si]\times [\Si']}^R_{g\times g'}.
		\end{align}
	\end{fact}
	In the positive direction, equality in (\ref{ineqprod}) holds in many well-known cases. 
	
	When $R=\R,$ by \cite[Theorem 5.2]{FMea}, we have,
	\begin{fact}\label{fctfm}
		If $d\le 2,$ or $c\le 2,$ or both $[\Si],[\Si']$ are of dimension $3$ and codimension $3,$ then we have
		$\no{[\Si]}^{\R}_{g}\times \no{[\Si']}^{\R}_{g'}=
		\no{[\Si]\times [\Si']}^{\R}_{g\times g'},
		$\end{fact}
	It is an open problem whether the above fact holds unconditionally for all dimensions and codimensions. See \cite{RZcf} for recent progress.
	
	Federer \cite{HFrf} proved that codimension one $\Z$-area-minimizing representatives on orientable ambient manifolds are also $\R$-area-minimizing representatives. Thus, by Fact \ref{fctfm} and properties of codimension $1$ differential forms, we have that,
	\begin{fact}\label{fctfedco}
		When both $[\Si],[\Si']$ are of codimension $1$ and both $M,M'$ are orientable,
		\begin{align}\label{eqprodz}
			\no{[\Si]}^{\Z}_{g}\times \no{[\Si']}^{\Z}_{g'}=		\no{[\Si]\times [\Si']}^{\Z}_{g\times g'}.
		\end{align}
	\end{fact}  
	In other words, the Cartesian product of $\Z$-area-minimizing hypersurfaces on orientable manifolds are $\Z$-area-minimizing on the Riemannian product of ambient manifolds. Tracing through the proof, one can show that Fact \ref{fctfedco} holds for successive products of more than two homology classes as well.
	
	Examples like products of K\"{a}hler subvarieties in products of K\"{a}hler manifolds \cite{HL} show that (\ref{eqprodz}) hold in other codimensions in natural settings. By \cite{FMcalv} and \cite[Lemma 2.11]{ZLa}, holomorphic subvarieties of complex flat tori show that (\ref{eqprodz}) holds with $\Z$ replaced by $\Z/n\Z$ for $n\in\Z_{\ge 2}$ in natural settings as well. 
	
	On the other hand, for Plateau problems in Euclidean space, \cite{FMmc,BWmc} provide examples where Cartesian products of solutions are not solutions to the product boundary conditions. As another byproduct of Theorem \ref{thmvn}, we prove that products of area-minimizing submanifolds are not generically area-minimizing, thus quantifying the failure of the K\"{u}nneth formula at sending $\Z$-area-minimizing representatives to $\Z$-area-minimizing representatives. 
	
	Recall Assumptions \ref{assumpm} and \ref{assumpmp}. First of all, if $[\Si]$ is a torsion class, then in many cases, $[\Si]'$ will annihilate $[\Si]$ in their Cartesian product, i.e., $[\Si]\times[\Si']=0$. To be precise, if $\ka[\Si]=0$, then $[\Si]\times[\Si']=0$ if $[\Si']=\ka[\Si'']$ for another class $[\Si'']$. In these cases, it holds trivially that the Cartesian products fails to send $\Z$-area-minimizing representatives to $\Z$-area-minimizing representatives. The same reasoning holds for $[\Si'].$ Thus, we will only consider the case of both factors being non-torsion classes. 
	\begin{thm}\label{thmkun}
		If $d\ge 1,c\ge 2$, $[\Si]$ is a non-torsion class on an orientable $M$, then for any real number $\ka>1,$ any metric $g'$ on $M'$, there exists a non-empty open set $\Om_{[\Si],g',\ka}$ in the space of Riemannian metrics on $M$, a non-empty open set $\Om_{[\Si],g',\ka}'\ni g'$ in the space of Riemannian metrics on $M'$, and  an integer $\nu_{[\Si],g',\ka}\ge 2$, such that for any metrics $h\in \Om_{[\Si],g',\ka}$, $h'\in \Om_{[\Si],g',\ka}'$ and any non-torsion homology class $[\si']\in \nu H_m(M',\Z)$ with  any integers $\nu\ge \nu_{[\Si],g',\ka},$ $0\le m\le d'+c',$ we have,
		\begin{align}
			\no{[\Si]}_h^{\Z}\times \no{[\si']}_{h'}^{\Z}> \ka \no{[\Si]\times[\si']}_{h\times h'}^{\Z}.
		\end{align}
	\end{thm}
	We want to emphasize that in the above statement $[\si']$ can be any non-torsion class of any dimension on $M'$.	It is an open problem \cite{FMmc} whether the above can hold among $R$-flat chains with $R=\R$ or $\Z/p\Z$ with $p$ a prime number.
	\section{Basic Definitions and Technical Preliminaries}\label{secbd}
	We will give some basic definitions and technical preliminaries in this section. The first four subsections are more or less standard definitions, and the experienced reader can skip them. The deeper technical preliminaries are in the remaining subsections.
	\subsection{Basics of Homology}
	Recall Assumption \ref{assumpm} and Assumption \ref{assumpmb}. We regard homology groups $H_\ast(M,\Z)$ as the relative homology groups $H_\ast(M,\es,\Z).$ 
	
We will later show in Lemma \ref{fcthomr} that $H_d(M,W,\Z)$ is always finitely generated. 
Thus, we have as a splitting up to choices of basis,
\begin{align}\label{decomphb}
	H_d(M,W,\Z)=\Z^{b_d}\oplus\bigoplus_{i\in I}\Z/p_i^{\nu_i}\Z,
\end{align}where $b_d$ is rank of $H_d(M,W,\Z)$, $I$ is a finite set of indices, the $p_i$ are not necessarily distinct prime numbers and $\nu_i\in \Z_{\ge 1}$. Again, by \cite[Theorem 1.57(c)]{JMgt}, the integers $\{b_d\}\cup_{i\in I}\{p_i,v_i\}$ and the subgroup $\bigoplus_{i\in I}\Z/p_i^{\nu_i}\Z$ are the canonically defined and independent of the choice of basis. \begin{defn}\label{defntmb}
	Use $G_d(M,W,\Z)$ to denote the torsion subgroup of $H_d(M,W,\Z)$, i.e., 
	\begin{align*}
		G_d(M,W,\Z)=\bigoplus_{i\in I}\Z/p_i^{\nu_i}\Z.
	\end{align*}Define the $d$-th torsion number $\tau_d$ of $(M,W)$ to be
	\begin{align*}
		\tau_d=\operatorname{lcm}_{i\in I}p_i^{\nu_i},
	\end{align*}where $\operatorname{lcm}$ means the 
	least common multiple. Set $\tau_d=1$ if $H_d(M,W,\Z)$ is torsion-free. 
\end{defn}
Now fix a decomposition  of $H_d(M,W,\Z)$ into subgroups in (\ref{decomphb}), and construct a basis of $H_d(M,W,\Z)$ with respect to this decomposition. 
	\begin{assump}\label{decompbs}
		Let $\{f_1,\cd,f_{b_d},v_1,\cd,v_{|I|}\}$ be a basis of $H_d(M,W,\Z)$ so that 
		\begin{itemize}
			\item 	$\{f_1,\cd,f_{b_d}\}$ is a $\Z$-linear independent basis of the factor $\Z^{b_d}$ in (\ref{decomph}),
			\item for all $j\in I,$ the element $v_j$ generates the $\Z/p_j^{\nu_j}\Z$ factor in (\ref{decomph}).
		\end{itemize}
	\end{assump}We need the following lemma
	\begin{lem}\label{lemnt}
		If $X,Y\in H_d(M,W,\Z)$ are two integral homology classes with the same $\Z/n\Z$-reduction, i.e., $X^{\Z/n\Z}=Y^{\Z/n\Z}$, then
		\begin{align*}
			X-Y\in nH_d(M,W,\Z).		\end{align*} Furthermore, when $\tau_d|n$, then $X-Y$ is $n$ times a non-torsion class.  
	\end{lem}
	\begin{proof}
		It suffices to determine the kernel of $\Z/n\Z$-reduction. By \cite[Section 3.2 p. 215 property (5)]{AH}, the kernel of $\ts \Z/n\Z$ of an abelian group is $n$ times the group. Thus, if $X,Y$ has the same $\Z/n\Z$-reduction, we have
		\begin{align*}
			X-Y\in nH_d(M,W,\Z).
		\end{align*}
		If $\tau_d|n,$ then $nH_d(M,W,\Z)$ contains no nonzero element of the torsion subgroup $G_d(M,W,\Z).$ We are done.
	\end{proof}
	We also need a classical result by Thom \cite{RT}.
	\begin{lem}\label{lemthom}
		With the same assumptions as in Theorem \ref{thmvn}, there exists an infinite subset $K\s \Z_{\ge 2},$ such that for all $k\in K$, the integral homology class $k[\Si]$ can be represented by a smoothly embedded connected compact submanifold,
	\end{lem}We want to emphasize that here we care about embedded submanifolds representing homology classes, not Steenrod realization by continuous maps.
	\begin{proof}
First, le us prove the lemma without the connectedness assumption. Let us argue by contradiction. Suppose there exists only a finite subset $F\s \Z_{\ge 2}$ so that for $k\in F$, the integral homology class has a smoothly embedded representative. By \cite[Theorem II.29]{RT} some multiple of $\left(\prod_{k\in F}k\right)[\Si]$ admits a smoothly embedded representative, a contradiction. Here we regard $\prod_{k\in F}k$ as $1$ in case $F$ is empty.
				
Next, let us show that when the codimension of $[\Si]$ with respect to the ambient manifold $M$ is at least two, any smoothly embedded representative of $k[\Si]$ with finitely many connected components can be made into a connected submanifold by connected sum of submanifolds. The case of $d\ge 2$ is classical \cite{CL} and a detailed write up can be found in \cite[Lemma 2.8.2]{ZLns}. For $d=1$, we utilize a direct construction. For two disjoint smoothly embedded circles $\ga_1,\ga_3$, we take another curve $\ga_2$ that starts from a point on $\ga_1$ and end at a point on $\ga_3.$ Let $\ga_4$ be $\ga_2$ reversely traversed. Then $\ga_1+\ga_2+\ga_3+\ga_4$ is homologically equivalent to $\ga_1+\ga_3$. As embeddings are generic maps when $d=1<2\le c$ \cite[2.13 Theorem, 2.6 Theorem]{MH}, a generic perturbation of $\ga_1+\ga_2+\ga_3+\ga_4$ gives a closed embedded circle $\ga_1\#\ga_3\in [\ga_1]+[\ga_3]\in H_1(M,\Z).$ 
	\end{proof}
For readers who prefer constructive proof, instead of the arguing by contradiction in the proof of Lemma \ref{lemthom}, when $d<c$, we can use transversality and successive connected sums of one smooth representative to get infinitely many classes with connected embedded representatives. However, the constructive argument won't work when $d>c$ as there is no general way to desingularize transverse intersections. 
\subsection{Balls and spheres}
For a Riemannian manifold $M^{d+c}$, we will reserve the boldface notation $\mathbf{B}^{d+c}_r(p)$ to denote the closed geodesic ball of radius $r$ centered at $p$  and $\mathbf{S}^{d+c}_r(p)$ to denote the closed geodesic sphere of radius $r$ centered at $p$.

When we do not impose a base point, the notations $\mathbf{B}^d_1$ and $\mathbf{S}^{d-1}_1$ means the standard unit ball of $\R^{d+1}$ and the standard unit sphere of $\R^{d}$ regarded as Riemannian manifolds. We need the following classical fact \cite[p.208 Exercise 4]{SSfourier}
\begin{fact}\label{fctvols}
	We have
	\begin{align*}
		\mathcal{H}^{d-1}(\mathbf{S}^{d-1})=&d\frac{\Ga\left(\frac{1}{2}\right)^d}{\Ga\left(\frac{d}{2}+1\right)},\\
		\mathcal{H}^{d}(\mathbf{B}^{d-1})=&\frac{\Ga\left(\frac{1}{2}\right)^d}{\Ga\left(\frac{d}{2}+1\right)}.
	\end{align*}
\end{fact}
	\subsection{Flat Chain and Currents}
	Let us recall the discussion about flat chains before Definition \ref{defnfc}. However, the space of flat chains is too large to have nice compactness properties. To this end, we need a smaller space. By \cite[Sections 4.2.16 and (4.2.16)$^v$ of 4.2.26]{HF}, we have, 
	\begin{fact}\label{fctrc}
		For $R=\Z$ or $\Z/n\Z$ for $n\in\Z_{\ge 2}$, finite mass $R$-flat chains form the space of $R$-rectifiable currents. 
	\end{fact}
	By \cite[Sections 4.2.17 and (4.2.17)$^v$ of 4.2.26]{HF}, the space of $R$-rectifiable currents with finite boundary mass enjoys compactness and closure properties that allow us to find $R$-area-minimizing representatives. Since we are concerned with area-minimization, it suffices to consider only $R$-rectifiable currents from now on for these two discrete $R.$
	
	An alternative characterization for $\Z$-rectifiable currents \cite[Section 27]{LS} is via the dual space of differential forms. \newcommand{\rct}{\operatorname{Rect}}
	\begin{fact}\label{fctcf}
		Every $d$-dimensional $\Z$-rectifiable current $T$ 	belongs in the dual space of smooth differential forms, with actions on a form $\phi$ defined by
		\begin{align*}
			T(\phi)=\int_{\rct T}\ta_T\phi(\eta)d\mathcal{H}^d,
		\end{align*}where $\rct T$ is the Hausdorff $d$-dimensional (countably) underlying rectifiable set of $T$ defined in \cite[Section 4.1.28(4)]{HF}, $\ta_T$ is the density of the mass measure of $T$ and $\eta$ is the simple unit $d$-vector representing the oriented tangent space to $\rct T$ Hausdorff $d$-dimensional a.e.~
	\end{fact}	Caution is needed for the notion of underlying rectifiable set and the notion of support. The support of an current is the complement of the set where the current vanishes \cite[Section 4.1.1]{HF}, which a priori might be larger than the underlying rectifiable set, e.g., countably many concentric circles with summable radius. However for $R$-area-minimizing flat chains, the two notions are the same by the monotonicity formula \cite[5.4.3 Theorem]{HF}, so we can use them interchangeably when discussing supports of $R$-area-minimizing flat chains.

	Conversely, we can also characterize $\Z$-flat chains using $\Z$ rectifiable currents. By \cite[bottom of p.382 Section 4.1.24]{HF}, we have,
	\begin{fact}\label{fctfcr}
		Every	$d$-dimensional $\Z$-flat chain decomposes as a sum of a $d$-dimensional $\Z$-rectifiable current and the boundary of a $(d+1)$-dimensional $\Z$-rectifiable current and vice versa. 
	\end{fact}
	
	The situation with $\R$-flat chains is a bit more complicated. Analogues of Fact \ref{fctrc} and the corresponding compactness and closure theorems for $\R$-rectifiable currents no longer hold. However, by \cite{HFrf}, the problem of finding $\R$-area-minimizing representatives can still be solved in $\R$-flat chains.
	
	An important fact about $\Z/n\Z$ rectifiable current is as follows \cite[p.430 of Section 4.2.26]{HF}. 
	\begin{fact}\label{fctrmn}
		For every $\Z/n\Z$-rectifiable current $T,$ there exists at least one $\Z$-rectifiable current $V$, called Federer's ``representative" modulo $n$ of $T$ such that the $\Z/n\Z$-mass measure of $T$ coincides with the $\Z$-mass measure of $V$ and $T=V^{\Z/n\Z}.$
	\end{fact}
	Unfortunately, Federer used the word ``representative" very differently from how we did. In order not to cause reader frustration, we will always use the phrase Federer's ``representative" modulo $n$ together.
	
	We want to emphasize that\begin{fact}\label{fctfce}In general,
		Federer's ``representative" modulo $n$
		\begin{itemize}
			\item  is not unique and many times there are infinitely many,  
			\item of a $\Z/n\Z$-cycle may not even be a $\Z$-cycle at all. 
		\end{itemize}
	\end{fact}
	For instance, consider an embedded real projective plane $\mathbb{RP}^2$ inside the four-sphere $S^4.$ For each homotopically non-trivial embedded $C^1$ curve $\ga$ on $\mathbb{RP}^2$, cut $\mathbb{RP}^2$ along $\ga$ to obtain an orientable domain $\De.$ Then there exists a unique orientation of $\De$ so that $\pd \De=2\ga$ as $\Z$-rectifiable currents. Any such $\De$ is a Federer's ``representative" modulo $2$ of $\mathbb{RP}^2$, and thus illustrates both bullets of Fact \ref{fctfce}.
	\subsection{Definition of regular and singular sets}
		\begin{defn}\label{defnsm}
	We say an $R$-rectifiable current $T$ is smooth at a point $p$ in the support of $T$ if there exists an open set $U$ containing $p$ on $M$, such that  $T$ restricted to $U$ equals $\ai$ times a smooth submanifold $N$, with $\ai\in R.$ The definitions of regular sets and singular sets of $T$ are as follows:
	\begin{itemize}
		\item The singular set of $T,$ $\sing T,$ is defined as set of points in the support of $T$ where $T$ is not smooth,		
		\item 
		The regular set of $T,$ $\operatorname{Reg}T$, is defined as the set of points in the support of $T$ where $T$ is smooth. 
	\end{itemize}
	\end{defn} 
	Here, support means the underlying set of an integral current defined in \cite[Section 4.1.28(4)]{HF}. From now on, we will also use the symbol $\supp T$ to mean the support of $T.$ 
	For example, the figure $8$ has a singular point at its self-intersection, while figure $0$ counted with multiplicity $-2$ is smooth as $\Z/5\Z$-rectifiable current.
	
	We want to emphasize here that all points $p\in\supp \pd T$ are singular points. In other words, regular sets only contain interior regular points of $T$.
	\subsection{Almgren stratification of singular sets}\label{ssas}
	Recall the definition of $d$-dimensional stationary varifolds in \cite{WA}. Roughly speaking, they can be understood as mass measures of unoriented countably rectifiable sets counted with positive integer multiplicities, and are critical points of $d$-dimensional area. As $\Z$-area-minimizing representatives and $\Z/n\Z$-area-minimizing representatives, with $n\in \Z_{\ge2}$, are minimizers of $d$-dimensional area, their associated mass measures are naturally stationary varifolds.
	
	For a $d$-dimensional stationary varifold $T$, the Almgren stratification of $ T$ (\cite{BWst}) 
is an ascending chain of closed subsets of the support of  $T$
	\begin{align*}
	\mathcal{S}_0T\s\mathcal{{S}}_1T\s\cd\s\mathcal{S}_dT=\supp T,\end{align*}
	such that $\mathcal{S}_j T$ consists of points in $\supp T$ with tangent cones having at most $j$-dimensional translational invariance. The reader can just regard the Almgren stratification as a geometric measure theory analogue of Whitney stratification.

	Define the $j$-th stratum of the Almgren stratification of $T$ as points in $\supp T$ who has a tangent cone with precisely $j$-dimensional translation invariance. In other words,
	\begin{defn}
	We call $\mathcal{S}_jT\setminus\mathcal{S}_{j-1}T$ the $j$-th stratum of the Almgren stratification, with $S_{-1}T=\emptyset.$
	\end{defn}
	For integers $j\le d-1,$ points in the $j$-th stratum of $T$ necessarily lie in the singular set of $T.$ 
	
	The classical regularity result in \cite{BWst} is
	\begin{fact}\label{fctas}
	The $j$-th Almgren stratification of $T$ has Hausdorff dimension at most $j.$
	\end{fact}
\begin{defn}
		We define the Almgren stratification of any subset $S$ of $\supp T$ by setting 
	\begin{align*}
		\mathcal{S}_jS=\mathcal{S}_jT\cap S.
	\end{align*}
\end{defn}		\subsection{The obstruction to lift $\Z/n\Z$ chains to $\Z$}\label{ssreg}
	In this subsection, we will prove an elementary lemma about the regularity of $\Z/n\Z$ area-minimizing representatives. It isolates the obstruction to lift a $\Z/n\Z$-flat chain cycle to a $\Z$-flat chain cycle. Recall Fact \ref{fctrmn}.
	\begin{lem}\label{lemdv}
		Suppose $T$ is a $d$-dimensional $\Z/n\Z$-flat chain on a not necessarily compact $(d+c)$-dimensional Riemannian manifold $M^{d+c}$. Let $S$ be a rectifiable current that is a Federer's ``representative" modulo $n$ of $T$. 
		Then
		\begin{align*}
			\spt\pd S\s \left(\operatorname{Reg}(T)\cap\left\{p\bigg|\ta_T(p)<\frac{n}{2}\right\}\right)\cp,
		\end{align*}where $\operatorname{Reg}(T)$ denotes the regular part of $T.$\end{lem}
	Here $\ta_T(p)$ is the density of $T$ at $p.$ 
	\begin{proof}
		Note that regular set of $T$ can have density at most $\frac{n}{2}$. We will prove that if $p\in\spt \pd S\cap \textnormal{Reg}(T),$ then $\ta_T(p)= \frac{n}{2}.$ 
		
		Since $p$ is in the regular set of $T,$ there exists a neighborhood $\br(p)$ in which $T$ restricted to $\br(p)$ equals $\ta_T(p) N$ for some smooth submanifold $N$ of $\br(p)$, so that in some coordinate system $(x_1,\cd,x_{d+c})$, $N$ is a smooth ball contained in the $x_1\cd x_d$-plane. By \cite[Corollary 1.4]{RY}, there exists a $\Z$-flat chain $R$ in $\br(p)$ so that
		\begin{align}\label{eqsrn}
			S\res \br(p)=nR+\ta_T(p) N,
		\end{align} as a $\Z$-rectifiable current. Here $S\res \br(p)$ means $S$ restricted to $\br(p).$ Since $S,N$ are both of finite mass, we deduce that $R$ is also of finite mass, thus a $\Z$-rectifiable current by \cite[Section 4.2.16]{HF}.
		
		If $R=0,$ then $\pd S$ restricted to $\br(p)$ is empty, as $S$ coincides with $\ta_T(p)N.$ Thus, we only have to deal with the case of $R\not=0.$
		
		Recall Fact \ref{fctcf}.
		There exist a $d$-dimensional rectifiable set $B$, an $\hm^d$ measurable simple unit $d$-vector field $\eta$ , so that the action of the current $R$ on differential forms is
		\begin{align*}
			R(\phi)=\int_B\ta_R\phi(\eta) d\hm^d(x), 
		\end{align*}
		with $\phi$ any smooth $d$-dimensional form. 
		
		By \cite[page 430]{HF}, the support of $S$ restricted to $\br(p)$ is contained in $N,$ thus by (\ref{eqsrn}) the support of $R$ is also contained in $N,$ i.e., $B\s N.$ This implies that $\eta=\pm \T_x N$ Hausdorff $d$-dimensional a.e.~ Here $\T_x N$ means the unique simple $d$-vector field representing the tangent space at $x$ to $N.$
		
		Thus, we have
		\begin{align*}
			S\res \br(p)(\phi)=\int_{N\setminus B}\ta_T(p) \phi(\T_xN)d\hm^d(x)+\int_{B} (\ta_T(p)\pm n\ta_R(x))\phi(\T_xN)d\hm^d(x).
		\end{align*} 
		By \cite[Section 4.1.28]{HF}, we deduce that
		\begin{align*}
			\ms^{\Z}(S\res \br(p))=\int_{N\setminus B}\ta_T(p)d\hm^d(x)+\int_B |\ta_T(p)\pm n\ta_R(x)|d\hm^d(x).
		\end{align*}
		However, since we have $R\not=0$, i.e., $\ta_R(x)$ being a positive integer on $B$ Hausdorff $d$-dimensional a.e.~, and $\ta_T(p)\le \frac{n}{2}$  by \cite[page 430]{HF}, we have 
		\begin{fact}\label{fctdstin}
			The inequality		\begin{align}\label{eqtatr}
				|\ta_T(p)\pm n\ta_R(x)|\ge \frac{n}{2},
			\end{align}holds
			where equality is achieved if and only if $\ta_T(p)=\frac{n}{2},\ta_R(x)=1.$
		\end{fact}	
		Thus, we have
		\begin{align*}
			\ms^{\Z}(S\res \br(p))\ge \ta_T(p) \ms(N)=\ms^{\Z/n\Z}(T\res \br(p)). 
		\end{align*} By Fact \ref{fctrmn}, we  have\begin{align*}
			\ms^{\Z/n\Z}(T\res \br(p))=	\ms^{\Z}(S\res \br(p)). 
		\end{align*}
		In other words, equality holds in (\ref{eqtatr}) Hausdorff $d$-dimensional a.e.~ on $B.$ By Fact \ref{fctdstin}, we deduce that
		$$\ta_T(p)=\frac{n}{2},\ta_R(x)=1,$$ on $B$ Hausdorff $d$-dimensional a.e.~ Thus $\pd S$ intersect $\reg T$ at $\ta_T=\frac{n}{2}$ and we are done.
	\end{proof} 
	\subsection{The regularity lemma}
	\begin{lem}\label{lemsingt}
	Suppose $T$ is a $d$-dimensional  $\Z/n\Z$-area-minimizing $\Z/n\Z$-	flat chain on a not necessarily compact $(d+c)$-dimensional Riemannian manifold $(M^{d+c},h)$ with possibly non-zero boundary $\pd T$.	Then	$$\sing T\cap (\supp\pd T)\cp\cap \{\ta_T(p,h)<\frac{n}{2}\}$$ has Hausdorff dimension at most $(d-2).$
	\end{lem}
	\begin{proof}
		Recall Section \ref{ssas}. Since $\ta_T(p,h)$ is an upper-semicontinuous function of $p,$ the set $\{\ta_T(p,h)<\frac{n}{2}\}$ is open. By \cite[Theorem 1.7]{DHMS}, the $d$-th stratum of the Almgren stratification of $\sing T\cap (\supp\pd T)\cp\cap \{\ta_T(p,h)<\frac{n}{2}\}$ is of Hausdorff dimension at most $(d-2)$. On the other hand, by \cite[Lemma 8.6]{DHMS}, for any point $p$ in the $(d-1)$-th stratum of the Almgren stratification of $\sing T\cap (\supp\pd T)\cp\cap \{\ta_T(p,h)<\frac{n}{2}\},$ we have $\ta_T(p,h)\ge\frac{n}{2}$, so the $(d-1)$-th stratum of the Almgren stratification of $\sing T\cap W\cap \{\ta_T(p,h)<\frac{n}{2}\}$ is empty. By Fact \ref{fctas}, we are done.
	\end{proof}
	The above regularity lemma will be applied together with the following fact
	\begin{fact}\label{fctbd}
Assume that $S$ is a $d$-dimensional $\Z$-rectifiable current on a not necessarily compact $(d+c)$-dimensional Riemannian manifold $M^{d+c}.$ If $$\supp\pd S\s W\cup O,$$ with $W$ and $W\cup O$ closed, $O$ disjoint from $W$ and $O$ having Hausdorff dimension at most $(d-2)$, then we have $$\supp \pd S\s W.$$ 
	\end{fact}
	\begin{proof}By Fact \ref{fctrmn}, $\pd S$ is a $\Z$-rectifiable current.
 By the definition of support \cite[Section 4.1.1 p.344]{HF}, it suffices to prove that for every point $p\in M\setminus W$, there is a neighborhood $U(p)$ in which $T(\phi)=0$ with $\phi$ any smooth $d$-form vanishing outside of $U(p).$ Recall Fact \ref{fctcf}. We have
\begin{align*}
	\pd T(\phi)=&\int_{\rct \pd T}\theta_{\pd T}\phi(\eta)d\mathcal{H}^{d-1}\\
	=&\int_{\rct \pd T\cap O}\theta_{\pd T}\phi(\eta)d\mathcal{H}^{d-1}+\int_{\rct  {\pd T}\cap W}\theta_T\phi(\eta)d\mathcal{H}^{d-1}\\
	=&0+0,
\end{align*}
 where at the last step we used $O$ having Hausdorff dimension at most $(d-2)$ and $\phi$ vanishing on $W$. We are done.
	\end{proof}
	\subsection{The norms across different metrics}\label{ssnorm}In this subsection, we will collect several facts about the norms $\no{\cdot}^R_g$ across different metrics. First of all, let us enlarge the definition of $\no{\cdot}^R_g.$ 
	\begin{defn}\label{defnrareaall}
	For any $R$-homology class $\omega\in H_d(M,R),$ define	\begin{align}\label{defnrall}
		\no{\om}^R_g=\inf_{\substack{T\in\mathcal{F}_g(M,R)\\ T\in\om}}\ms_g^R(T).
	\end{align}
For any subset $W$ of $H_d(M,R)$, define the set of $R$-area-minimizing representatives $\mn_g^R(W)$ of $W$ to be
	\begin{align*}
		\mn_g^R\left(W\right)=\left\{T\bigg|T\in \mathcal{F}_g(M,R),[T]\in W,\ms^R_g(T)=\inf_{\om\in W}\no{\om}^R_g\right\}.
	\end{align*} 
	\end{defn}
	Definition \ref{defnrarea} and Definition \ref{defnrareaall}  can be reconciled as
	\begin{align*}
	\no{[\Si]}_g^R=&\no{[\Si]^R}^R_g,\\
	\mn_g^R\left([\Si]\right)=&\mn_g^R\left([\Si]^R\right).
	\end{align*}
	Using the argument in \cite[34.5 Theorem]{LS}, we have
	\begin{fact}\label{fctnr}
	$\no{\cdot}^R_g$ depends continuously on $g.$
	\end{fact}
	For two different metrics $g,h$, by Definition \ref{defnrall}, we have,
	\begin{fact}\label{fctmca}
	If  $\ms_g^R\le A\ms_h^R$ for $A>0,$ then $\no{\cdot}_g^R\le A\no{\cdot}_h^R$. 
	\end{fact}
	Using \cite{HFrf} and \cite[Chapter 5]{HF}, we have
	\begin{fact}\label{factnorm}
	$\no{\cdot}^R_g$ is a norm in the sense of Definition \ref{defnnm}, i.e., satisfying reversibility, triangle inequality and positivity.
	\end{fact} 
	A quick corollary of Fact \ref{factnorm} is that
	\begin{fact}
	For any integer $k\in \Z$, $\om\in H_d(M,R),$ we have
	\begin{align}\label{inns}
		\no{k\om}_g^R\le |k|\no{\om}_g^R.
	\end{align}
	\end{fact}
	The inequality (\ref{inns}) can be strict in many situations, as we shall see in the proof of Theorem \ref{thmnc}.
	
	Now let us specialize to $R=\R$. Let us start with the basic facts \cite[Section 3.9 and Section 5.8]{HFrf}
	\begin{fact}\label{fctnrt}We have
	\begin{itemize}
		\item $\nog{\cdot}_g$ is turns the the vector space $H_d(M,\R)$ into a finite dimensional Banach space, 
		\item		$\nog{[\Si]}_g\not=0$  if and only if $[\Si]$ is a non-torsion class.
		\item we have
		\begin{align}\label{eqfed}
			\lim_{\substack{k\to\infty\\ k\in\N}}\frac{\no{k[\Si]}_g^{\Z}}{k}=\no{[\Si]}_g^{\R}.
		\end{align}
	\end{itemize}
	\end{fact} 
	A specific corollary of the first bullet in Fact \ref{fctnrt} is that $\no{\cdot}_g^R$ satisfies homogeneity, i.e., equality in (\ref{inns}) always holds for all  $k\in\R.$  
	By the triangle inequality, a quick corollary of the second bullet is that 
	\begin{fact}\label{fctfc}
	For any torsion class $\tau\in H_d(M,\Z)$, we have
	\begin{align*}
		\nog{\tau+[\Si]}_g=\nog{[\Si]}_g.
	\end{align*}
	\end{fact}
	
	When $R=\Z/n\Z$, we have some partial control on the homogeneity of $\no{\cdot}_{g}^{\Z/n\Z}$.
	\begin{lem}\label{lemhn}
	If an integer $k\in \left(-\frac{n}{2},\frac{n}{2}\right]$ invertible in $\Z/n\Z$ with its unique inverse in 	$\left(-\frac{n}{2},\frac{n}{2}\right]$ denoted by $l,$ then for any $\om\in H_d(M,\Z/n\Z),$ we have
	\begin{align*}
		|l|\m \no{\om}_g^{\Z/n\Z}\le\no{k\om}_g^{\Z/n\Z}\le |k|\no{\om}_g^{\Z/n\Z},
	\end{align*}
	and thus
	\begin{align}\label{inhn}
		\frac{2}{n}\no{\om}_g^{\Z/n\Z}\le\no{k\om}_g^{\Z/n\Z}\le \frac{n}{2}\no{\om}_g^{\Z/n\Z}.
	\end{align}
	\end{lem}
	\begin{proof}
	Apply (\ref{inns}) to $k\om$ and $l(k\om)$.
	\end{proof}
	We need a basic lemma about general norms.
	\begin{lem}\label{lemnorm}
	Fix a basis $\{v_1,\cd,v_b\}$ of $\R^b$. For any continuous Banach norm $\no{\cdot}$ on $\R^n,$ there exists a $C>0$ that depends continuously on $\no{\cdot},$ so that 
	\begin{align*}
		\sum_{j=1}^b|a_j|\no{v_j}\le C \no{\sum_{j=1}^b a_jv_j},
	\end{align*}with each $a_j\in \R.$
	\end{lem}
	\begin{proof}
	For $v=\sum_{j=1}^b a_jv_j$ with each $a_j\in\R$ define
	\begin{align*}
		\no{v}^\infty=\sum_{j=1}^b|a_j|\no{v_j}.
	\end{align*}
	It is straightforward to verify that $\no{\cdot}^\infty$ is a continuous function Banach norm on $\R^n$. We need to prove that $\frac{\no{\cdot}^\infty}{\no{\cdot}}$ is bounded. By homogeneity, it suffices to verify this on the norm unit sphere $\{v\in \R^n|\no{v}=1\}.$ This follows from continuity of $\no{\cdot}^\infty,$ and the the compactness of norm unit sphere.
	\end{proof}
	\begin{lem}\label{lemmsbd}
	There exists a continuous function $\upsilon$ mapping the space of Riemannian metrics	to $\R_{\ge 0},$ such that  \begin{align*}
		\sup_{[\Si]\in H_d(M,\Z)\textnormal{ is free}}		\frac{\no{[\Si]}_g^{\Z}}{\no{[\Si]}_g^{\R}}\le \upsilon(g).
	\end{align*}
	\end{lem}
	\begin{proof}
	Recall Assumption \ref{decompbs}. Every non-torsion class $[\Si]\in H_d(M,\Z)$ can be written as 
	\begin{align*}
		[\Si]=\sum_{j=1}^{b_d}\ai_jf_j+\sum_{i\in I}\be_iv_i,
	\end{align*}
	with each $\ai_j\in \Z,\be_i\in\Z\cap[0,p_j^{\nu_j}],$ and not all $a_i$ are $0.$ By Lemma \ref{lemnorm} there is a continuous function $C$ from the space of Riemannian metrics to $\R_{> 0}$ so that
	\begin{align*}
		\nog{[\Si]}_g=\no{\sum_{j=1}^{b_d}\ai_jf_j}^{\R}_g\ge C\m \sum_{j=1}^{b_d}|\ai_j|\no{f_j}^{\R}_g,
	\end{align*}where at the first equality we have used Fact \ref{fctfc}.
	By the mediant inequality and $\max_j|\ai_j|\ge 1$, we have
	\begin{align*}
		\frac{\no{[\Si]}_g^{\Z}}{\no{[\Si]}_g^{\R}}\le& C\frac{\sum_{j=1}^{b_d} |\ai_j|\no{f_j}_g^{\Z}+\sum_{i\in I}|\be_i|\no{v_i}^{\Z}_g}{\sum_{j=1}^{b_d}|\ai_j|\nog{f_j}_g}\\
		=&C\frac{\sum_{j=1}^{b_d} |\ai_j|\no{f_j}_g^{\Z}}{\sum_{j=1}^{b_d}|\ai_j|\nog{f_j}_g}+C\frac{
			\sum_{i\in I}|\be_i|\no{v_i}^{\Z}_g}{\sum_{j=1}^{b_d}|\ai_j|\nog{f_j}_g}\\
		\le&C\max_{1\le j\le b_d} \frac{\no{f_j}^{\Z}_g}{\no{f_j}_g^{\R}}+C\tau_b\frac{\max_{i\in I}\no{v_i}^{\Z}_g}{\min_{1\le j\le b_d}\no{f_j}_g^{\R}}.
	\end{align*}By Fact \ref{fctnr}, we are done.
	\end{proof}
		\subsection{Calibrations}\label{hscal}
	The standard way to prove $\Z$-area-minimizing and $\R$-area-minimizing is via calibrations. The notion of calibrations started with the classical paper of Harvey-Lawson \cite{HL}.
	
	Recall that the comass of a  $d$-dimensional differential form  $\phi$ is the maximum of $\phi$ evaluated on simple unit $d$-vectors in the tangent space to $M$ among all points \cite[Section 1.8]{HF}. In other words
	\begin{align}\label{defncms}
		\cms_g\phi=\max_{p\in M}\max_{\substack{P\s T_pM\\ \dim P=d}}\phi(P).
	\end{align}
	Here we use the symbol $P$ to denote both a $d$-dimensional plane $P$ and the unit simple $d$-vector representing $P.$
	
	For instance, if $\phi$ is a simple form, then the comass of $\phi$ is equal to the maximum Riemannian length of $\phi$ over all points. 
	
	Now we are ready to give the definition of calibrations.
	\begin{defn}\label{defncal}(Definition of calibrations)
		\begin{itemize}
			\item 	We say a closed $d$-dimensional Hausdorff measurable $d$-form $\phi$ on a (possibly open) ambient manifold is a calibration if its comass is at most $1.$ 
			\item 
			We say a $d$-dimensional integral current $T$ is calibrated by $\phi,$ if $d$-dimensional Hausdorff measure a.e.~where $\phi$ restricted to the tangent space of $T$ equals the volume form of $T.$
		\end{itemize}
	\end{defn}
	In general, if the calibration form $\phi$ has sufficient regularity, then any integral current $T$ calibrated by $\phi$ will be area-minimizing. For smooth forms, Harvey-Lawson \cite[Chapter II, Theorem 4.2]{HL} proved that
	\begin{fact}\label{fcthl}
		If a $d$-dimensional integral current $T$ on a Riemannian manifold is calibrated by a degree $d$ smooth calibration form $\phi$, then $T$ is both $\Z$-area-minimizing and $\R$-area-minimizing, and any integral current homologous to $T$ and having the same area as $T$ must also be calibrated by $\phi.$
	\end{fact}
	The widest possible notion of calibrations is found in $\R$-flat cochains \cite{HFrf}, which are defined as bounded continuous linear functionals on $\R$-flat chains. By \cite[Section 4.1.19]{FF}, $\R$-flat cochains are equivalent to measurable forms with weak exterior derivatives that both have finite mass.
	\begin{defn}\label{defncalr}
		We say an $\R$-flat cochain $\ai,$ is a calibration on $(M,g),$ if 
		\begin{itemize}
			\item $d\ai=0,$
			\item $\ai(T)\le \ms^{\R}_g(T),$ for all $\R$-flat chains $T.$
		\end{itemize}
		We say an $\R$-flat chain $T$ is calibrated by $\ai,$ if
		$$\ai(T)=\ms^{\R}_g(T).$$	
	\end{defn}
	By definition of calibrations, it is clear that any classical calibration form serves as a calibration $\R$-flat cochain via the canonical action of differential forms on $\R$-flat chains. Fact \ref{fcthl} holds when $T$ is an $\R$-flat chain and $\phi$ is a calibration $\R$-flat cochain.
	
	By \cite[Section 4.12]{HFrf}, we have
	\begin{fact}\label{fctcalcc}
		Every element of $\mn_g^{\R}\left([\Si]\right)$ is calibrated by a calibration $\R$-flat cochain.
	\end{fact}	
	By Fact \ref{fcthl} and \ref{fctcalcc}, we have
	\begin{fact}\label{fctcal}
		A $\Z$-area-minimizing integral current representing $[\Si]$ is calibrated by an $\R$-flat cochain if and only if $\no{[\Si]}_g^{\Z}=\no{[\Si]}_g^{\R}$.
	\end{fact}
	\subsection{Subanalytic geometry}In this subsection, we will recall several basic facts about compact $(d-1)$-dimensional subanalytic sets and $\ran$-definable sets in o-minimal geometry. 
	
	For readers not familiar with these concepts, it suffices to treat them as black boxes. Essentially only two properties are needed: existence of finite triangulation \cite[Theorem 4.4]{MComin} and uniform bound on intersections \cite[Exercise 3.13]{MComin}. Great introductions include \cite{MComin,VDDom}, where the following is taken from.
	
	From our perspective, the simplest examples of subanalytic sets to keep in mind are semialgebraic sets, i.e., finite unions and intersections of sets cut out locally by algebraic equations and inequalities. The class of semi-algebraic sets is closed under projections onto Euclidean subspaces, which on the propositional calculus side gives us quantifier eliminations (Tarski's theorem) and yields a very tame and well-behaved class of sets. Heuristically, for a semi-algebraic set, we have stratification into pieces of different dimensions and the largest possible dimension of these pieces is the dimension of the semi-algebraic set. For a $(d-1)$-dimensional semi-algebraic set in $\R^{d+c}$, we have uniform control on the number of intersections it can have with almost every $(c+1)$-dimensional affine planes and this is the most crucial property to us.

However, the sets we want to deal with have to non-algebraic compact analytic sets, e.g., branched minimal surfaces in spheres. Thus, we come to semianalytic sets, i.e., finite unions and intersections of sets locally cut out by analytic equations and inequalities. However, semianalytic sets do not form a very well-behaved class, i.e., no longer closed under projections, so we enlarge the class to subanalytic sets \cite{BMsub} to include projections of semianalytic sets. 

Unfortunately, the class of subanalytic sets are a bit too large, as we might have infinite intersections like zero set of $\sin$, i.e., the intersection of the graph of the function $\sin$ with the real line.
	
	Thus, on our ends, we need another notion that captures simultaneously the finite intersection properties and the analyticity appeared in minimal submanifolds. The right category to consider is $\ran$-definable sets in o-minimal geometry. Roughly speaking, instead of consider any analytic function, we only consider those restricted to a compact subset, e.g., a cube. Then regain uniform control on intersections.
	\begin{fact}\label{fctranl}\cite[Definition on p.506 and discussion right after it]{VDDMom}
		An $\ran$-definable set of $\R$ is a finite union of intervals and points.
	\end{fact}
$\ran$-definable sets are also called globally subanalytic, i.e., subanalytic in the projection completion \cite[2.5 Example (4)]{VDDMom},
	\begin{fact}\label{fctsubran}
		Compact subanalytic sets in Euclidean spaces are $\ran$-definable.
	\end{fact}
Recall Assumption \ref{assumpmb}. We need the following fact.\begin{lem}\label{fctsub}
	If $W$ is a $(d-1)$-dimensional compact subanalytic set of $\R^{d+c}$ and $V$ is a $(c+1)$-dimensional subspace of $\R^{d+c}$, with $c,d\ge 1,$ then for Haar measure almost every $g\in O(d+c)$, $W\cap g.V$ consists of at most $\mu_W$ points. Furthermore, affine translates $W+p$ of $W$ satisfies
	\begin{align*}
		\mu_{W+p}=W.
	\end{align*} 
\end{lem}	
Lemma \ref{fctsub} follows from the following fact
\begin{fact}\label{fctodc}
	The set of points in $g\in O(d+c)$ satisfying $\dim g.V\cap W\ge 1$ is $\ran$-definable and has dimension most $\dim(O(d+c))-1$
\end{fact}
\begin{proof}
The case of $d=1$ is trivial as $\dim W=0$. From now on we assume that $d\ge 2.$	By \cite[Exercise 3.13]{MComin} (which itself is just a successive application of \cite[Theorem 3.12]{MComin}), for every $g\in O(d+c),$ every affine translation $p+W,$ the intersection $g.V\cap (p+W)$ has at most $\mu_W$ connected components, which are $\ran$-definable sets. By Fact \ref{fctodc}, Haar almost every $g\in O(d+c),$ the intersection $g.V\cap W$ consists only of $0$-dimensional $\ran$-definable sets, i.e., discrete points. We are done.
\end{proof}
To prove Fact \ref{fctodc}, first we need an auxiliary lemma about a special $\ran$-definable set of $O(d+c).$
\begin{fact}\label{fctovv}
	Let $V$ be a $(c+1)$-dimensional subspace of $\R^{d+c}$ and $v$ a unit length vector in $V.$ Then the subset $O(V,v)$ of $O(d+c)$ defined by
	\begin{align*}
		O(V,v)=\{g|g\in O(d+c),gv\in V\}
	\end{align*}is a smooth $\ran$-definable set of dimension
	\begin{align*}
		\dim O(V,v)=\dim(O(d+c))-(d-1).
	\end{align*}
\end{fact}
\begin{proof}
	Choose a coordinate $(x_1,\cdots,x_{d+c})$ system on $\R^{d+c}$ so that $V$ is the coordinate plane $x_1\cdots x_{c+1}$ and $v$ is the coordinate vector $\pd_{x_1}$. Identify $\R^{(d+c)^2}$ with the space of matrices via entries.
	
	Consider the real algebraic map
	\begin{align*}
		F:\R^{(d+c)^2}\to \R^{\frac{(d+c)(d+c+1)}{2}}\times\R^{d-1},
	\end{align*}
	defined by
	\begin{align*}
		F(g)=(gg^t,g_{c+2,1},\cdots,g_{c+d,1}),
	\end{align*}	where $\R^{\frac{(d+c)(d+c+1)}{2}}$ is the the linear space of symmetric matrices, $t$ denotes transpose and $g_{i,j}$ means the $ij$-th entry of  $g.$ 
	
	Then $O(V,v)$ is the real algebraic subvariety of $O(d+c)$ defined by
	\begin{align*}
O(V,v)=F\m(I,0).
	\end{align*}
	Let us prove that $dF$ has rank $\frac{(d+c)(d+c+1)}{2}+(d-1)$  when $g\in O(V,v).$ Then by the implicit function theorem \cite[Theorem 7.9]{JL}, we deduce that $O(V,v)$ is a smooth submanifold of dimension
	\begin{align*}
		(d+c)^2-\frac{(d+c)(d+c+1)}{2}-(d-1)=\frac{(d+c)(d+c-1)}{2}-(d-1)=\dim(O(d+c))-(d-1).
	\end{align*}And we are done.

First, we calculate the differential of $F$ when $g$ is invertible.
	\begin{align*}
		dF_g(h)=&(hg^t+gh^t,h_{c+2,1},\cdots,h_{c+d,1})\\=&\left(g(g\m h+h^t(g^t)\m)g^t,h_{c+2,1},\cdots,h_{c+d,1}\right)\\=&\left(g(g\m h+(g\m h)^t)g^t,h_{c+2,1},\cdots,h_{c+d,1}\right).
	\end{align*}
	For any fixed $g$ with $\det g\not=0,$ do a linear change of variables as follows: left multiplication by $g$ on $\R^{(d+c)^2}$ and left multiplication by $g\m$ composed with right multiplication by $(g^t)\m$ on $\R^{\frac{(d+c)(d+c+1)}{2}}$,  we can rewrite
	\begin{align}\label{defndfn}
		dF_g(h)=(h+h^t,g_{c+2}.h^1,\cdots,g_{c+d}.h^1),
	\end{align}where $g_i.h^j$ denotes the inner product of the $i$-th row of $g$ with the $j$-th column of $h.$
	
Now let us specialize to $g\in O(V,v)=F\m(I,0)$.	Then $g_{c+2},\cdots,g_{c+d}$ are orthonormal frames of a vector subspace dimension $(d-1)$ and their first entries are all zero, i.e.,
\begin{align*}
	g_{c+2,1}=\cdots=g_{c+d,1}=0.
\end{align*}
Complete $g_{c+2},\cdots,g_{c+d}$ to an orthnormal frame of $\R^{d+c}:$
\begin{align*}
	w_1=(1,0,\cdots,0),w_2,\cdots,w_{c+1},g_{c+2},\cdots,g_{c+d}.
\end{align*}
Define $\kappa:\R^{d+c}\to\R^{(d+c)^2}$ by setting
\begin{align*}
	\kappa(w)=[w^t,0,\cdots,0]-[w^t,0,\cdots,0]^t.
\end{align*}
In other words, $\kappa(w)$ is the $(d+c)\times (d+c)$ matrix formed by regarding $-w$ as the first row vector and $w$ as the first column vector and add them to the zero matrix. By definition, 
\begin{align*}
	\kappa(w)^t=-\kappa(w).
\end{align*}We claim that 
\begin{fact}
$dF$ in (\ref{defndfn}) sends the span of the set of symmetric matrices and the set $\ka(\R^{d+c})$ onto $\R^{\frac{(d+c)(d+c+1)}{2}}$.
\end{fact}
To prove the above fact, let $e_i^j$ denote the $(d+c)\times(d+c)$ matrix with all entries zero outside of the $ij$-th entry, which is set to $1$. Then direct calculation shows that and we are done:
\begin{align*}
	dF(e_i^j+e_j^i)=&(2e_i^j+2e_j^i,0,\cdots,0)\textnormal{ for }i\ge 2,\\
	dF\left(e_1^j+e_j^1-\sum_{l=c+2}^{c+d}g_{l,j}\ka(g_l)\right)=&(2e_1^j+2e_j^1,0,\cdots,0)\textnormal{ for }1\le j\le c+d,\\
	dF\left(\ka(g_{j,1})\right)=&(0_{d+c})\times (0,\cdots,1,0,\cdots)\textnormal{ the entry }1\textnormal{ is at position }j-(c+1),
\end{align*}
where $0_{d+c}$ denotes the $(d+c)\times(d+c)$ $0$ matrix.
\end{proof}
Now we are ready to prove Fact \ref{fctodc}
\subsubsection{Proof of Fact \ref{fctodc}}We will use several times without mentioning again the citation that for $\ran$-definable sets $A,B$ \cite[Proposition 3.17 (3)]{MComin}, 
\begin{align}\label{eqdimu}
	\dim (A\cup B)\le \max\{\dim A,\dim B\}.
\end{align}
Le us consider the coincidence set. Define
\begin{align*}
	K=\{(g,x)|x\in g.V\cap W\}\s \R^{(d+c)^2}\times \R^{d+c}.
\end{align*}
Here we embed $O(d+c)$ into the set of $(d+c)\times (d+c) $ matrices, identified as Euclidean space $\R^{(d+c)^2}$ via entries. The condition of $x\in g.V\cap W$ is equivalent to $x\in W$ and $g\m x\in V.$ As $O(d+c)$ is a real algebraic variety, this implies that $K$ is the intersection of the $\ran$-definable set $O(d+c)\times W$ and $\{(g,x)|g\m x\in V\}.$ Thus, $K$ is also $\ran$-definable.

Roughly speaking, we will first prove that $K$ has dimension at most $\dim(O(d+c))$ and then use \cite[Theorem 3.18]{MComin} to deduce that the set of points $g\in O(d+c)$ satisfying $\dim g.V\cap W=l$ is $\ran$-definable and has dimension most $\dim(O(d+c))-l$.

Let us partition $K$ into a disjoint union,
\begin{align*}
	K=K_0\cup K_{\not0},
\end{align*}
by setting
\begin{align*}
	K_0=&K\cap\left( O(d+c)\times\{0\}\right),\\ K_{\not0}=&K\cap \left(O(d+c)\times W\setminus\{0\}\right).
\end{align*}
Both $K_0$ and $K_{\not0}$ are $\ran$-definable sets by construction.

If $0\in W$, then $K_0= O(d+c)\times\{0\}$. If $0\not\in W$, then $K_0=\es.$ Anyway, we always have
\begin{align*}
	\dim K_0\le\dim(O(d+c)).
\end{align*}
On the other hand, consider the projection $\pi_{\R^{d+c}}$ of $K_{\not 0}$ to the $\R^{d+c}$ factor. For any $x\in \pi_{\R^{d+c}}(K_{\not=0}),$ we have
\begin{align*}
\pi_{\R^{d+c}}\m (x)\cap K_{\not=0}=\{g|g\in O(d+c),x\in g.V\}.
\end{align*}
Let us determine the dimension of $\pi_{\R^{d+c}}\m (x).$ Since $x\in \pi_{\R^{d+c}}(K_{\not=0}),$ there exists a $g_x$ so that $$v_x=g_x.x\in V.$$ Thus, 
\begin{align*}
\pi_{\R^{d+c}}\m (x)\cap K_{\not=0}=\{g|g\in O(d+c),g.\frac{v_x}{|v_x|}\in V\}.g_x\m.	
\end{align*}
Apply Fact \ref{fctovv} to $\{g|g\in O(d+c),g.\frac{v_x}{|v_x|}\in V\}$, we deduce that
\begin{align*}
	\dim \pi_{\R^{d+c}}\m (x)\cap K_{\not=0}=\dim(O(d+c))-(d-1).
\end{align*}
In other words, all slices of $K_{\not0}$ by $\pi_{\R^{d+c}}$ has exactly the same dimension $O(d+c)-(d-1)$. Thus, by \cite[Theorem 3.18, $X_d=\pi_{\R^{d+c}}(K_{\not 0}),d=\dim(O(d+c))-(d-1),m=d+c,n=(d+c)^2$]{MComin}, we have
\begin{align*}&\dim(K_{\not0})\\=&
\dim\left(K_{\not0}\cap \left(\R^{(d+c)^2}\times\pi_{\R^{d+c}}(K_{\not 0})\right)\right)\\=&\dim(O(d+c))-(d-1)+\dim \pi_{\R^{d+c}}(K_{\not 0})\\\le& \dim(O(d+c))-(d-1)+\dim W\\
=&\dim(O(d+c)).
\end{align*}
Thus, we have
\begin{align*}
	\dim K\le \max\{\dim K_0,\dim{K_{\not0}}\}\le \dim(O(d+c))
\end{align*}
For $0\le l\le \min\{\dim P,\dim W\},$ define\begin{align*}
	S_{\dim g.V\cap W= l}=&\{g|g\in O(d+c),\dim g.V\cap W=l\}=\{g|g\in O(d+c),\dim\pi_{\R^{d+c}}(K)=l\},\\
	S_{\dim g.V\cap W\ge 1}=&\cup_l S_{\dim g.V\cap W= l}.
\end{align*}
In other words, $S_{\dim g.V\cap W= l}$ is the set of elements in $O(d+c)$, where $g.V\cap W$ is of dimension $l,$ or equivalently, $\pi_{\R^{d+c}}(K)$ has dimension $l.$ By \cite[Theorem 3.18]{MComin}, $S_{\dim g.V\cap W= l}$ is $\ran$-definable for all $l.$

By \cite[Theorem 3.18, $A=S_{\dim g.V\cap W= l},d=l,m=(d+c)^2,n=d+c$]{MComin}, for $0\le l\le \min\{\dim P,\dim W\}$, we have
\begin{align*}
	\dim S_{\dim g.V\cap W= l}+l=\dim(K\cap (S_{\dim g.V\cap W= l}\times\R^{d+c}))\le\dim K\le\dim(O(d+c)).
\end{align*}
In other words,
\begin{align*}
	\dim S_{\dim g.V\cap W= l}\le \dim(O(d+c))-1.
\end{align*}
This implies that
\begin{align*}
	\dim S_{\dim g.V\cap W\ge 1}\le O(d+c)-1.
\end{align*}
We are done.
			\subsection{Constancy theorem and Homology of $R$-flat chains in $W$}
	We will use the following constancy theorem several times
	\begin{fact}\label{fctct}
		Let $N$ be a $d$-dimensional open submanifold in a complete $(d+c)$-dimensional manifold $M$. Then an $d$-dimensional $R$-flat chain $T$ such that $\supp \pd T\cap N=\es$ restricted to $N$ equals $\ai N$ for some $\ai\in R.$
	\end{fact}
	\begin{proof}
		Apply \cite[6.1 Theorem]{MSfc}
	\end{proof}
	Recall Assumption \ref{assumpmb}. Using the above Fact \ref{fctct}, let us first prove that the $d$-th homology group of $R$-flat chains supported in $W$ is isomorphic the $d$-dimensional simplicial homology of $W$.
	\begin{lem}\label{fcthomr}
		The $d$-dimensional $R$-homology of $W$, $H_d(W,R)$, satisfies the following properties
		\begin{enumerate}
			\item $H_{d-1}(W,R)$ is finitely generated abelian group and is free $R=\Z$ or $\R$,
			\item every element of $H_{d-1}(W,R)$ can be represented by a unique $(d-1)$-dimensional simplicial $R$-cycle and vice versa,
			\item every $(d-1)$-dimensional $R$-flat chain cycle supported in $W$ can be represented by a simplicial cycle, 
			\item the $(d-1)$-dimensional homology of $R$-flat chains supported in $W$ is isomorphic to $H_{d-1}(W,R).$ 
			\item every element of $H_{d-1}(W,R)$ can be represented by a unique $(d-1)$-dimensional $R$-flat chain cycle and vice versa.
		\end{enumerate}
		\end{lem}
		\begin{proof}
			First, $W$ admits a finite triangulation (\cite[7.4 Theorem]{BPhs} for subanalytic $W$ and \cite{SCtr} for smooth $W$). Next, there is no $d$-dimensional simplices, so there is non-nontrivial exact $(d-1)$-dimensional chains. Thus, any homology class of $H_{d-1}(W,R)$ has a unique representative by a simplicial cycle. This means that the group $H_{d-1}(W,R)$ is a subgroup of the finitely generated abelian group of $R$-simplicial $d$-dimensional chains in $W$. By \cite[Proposition 4.52]{AKba} $H_{d-1}(W,R)$ is finitely generated. When $R=\Z$ or $\R$, the group of $R$-simplicial chains is free, so its subgroup $H_{d-1}(W,R)$ is also free \cite[Theorem 4.55]{AKba}. We are done with the first bullet and the second bullet.
			
			For the third bullet, as $\pd T=0$, by Fact \ref{fctct}, every $R$-flat cycle is represented by an simplicial $R$-chain and then apply the second bullet. 
			
			The fourth bullet and the last bullet follows from the third bullet and the second bullet.	
		\end{proof}
	\subsection{Plateau problem across different norms}
	The discussion in Section \ref{ssnorm} applies to Plateau problems in Euclidean space as well. Recall Definition \ref{defnpf}. The results Section \ref{ssnorm} holds over verbatim with homology replaced by relative homology. Let us reiterate the details.
	
Again, let us first enlarge the definition of $\no{\cdot}^R.$ 
	\begin{defn}\label{defnrareaallb}
Use $\mathcal{C}(\R^n,R)$ to denote the subspace of $\R$-flat chains on $\R^n$ that are cycles, i.e., \begin{align*}
	\mathcal{C}(\R^n,R)=\{T|T\in\mathcal{F}(\R^n,R),\pd T=0\}.
\end{align*}		For any $R$-flat chain $\beta\in \mathcal{C}(\R^n,R)$ define	\begin{align}\label{defnrallb}
			\no{\be}^R_\de=\inf_{\substack{T\in\mathcal{C}(\R^n,R)\\ \pd T=\be}}\ms^R_\de(T).
		\end{align}
For any subset $S$ of $\mathcal{C}_d(\R^{d+c},R)$, define the set of $R$-area-minimizing representatives $\mn^R_\de(S)$ of $S$ to be
\begin{align*}
	\mn^R_\de\left(S\right)=\left\{T\bigg|T\in\mathcal{F}(\R^n,R),\be\in S,\ms^R_\de(T)=\inf_{\be\in S}\no{\be}^R_\de\right\}.
	\end{align*} 	\end{defn}	
	Definition \ref{defnpf} and Definition \ref{defnrareaallb}  can be reconciled as
	\begin{align*}
		\no{[\Si]}^R_\de=&\no{\pd[\Si]^R}^R_\de.
	\end{align*}
	Using \cite{HFrf} and \cite[Chapter 5]{HF}, we have
	\begin{fact}\label{factnormb}
		$\no{\cdot}^R_\de$ is a norm in the sense of Definition \ref{defnnm} on $\mathcal{C}(\R^n,R)$ i.e., satisfying reversibility, triangle inequality and positivity.
	\end{fact} 	
		A quick corollary of Fact \ref{factnormb} is that
	\begin{fact}
		For any integer $k\in \Z$, $\be\in\mathcal{C}(\R^n,R),$ we have
		\begin{align}\label{innsb}
			\no{k\be}^R_\de\le |k|\no{\be}^R_\de.
		\end{align}
	\end{fact}
	The inequality (\ref{innsb}) can be strict in many situations, as proven by Morgan \cite{FMmc} and White \cite{BWmc}.
	
	Now let us specialize to $R=\R$ and start with the basic facts \cite[Section 3.9 and Section 5.8]{HFrf}.	\begin{fact}\label{fctnrtb}We have
		\begin{itemize}
					\item $\nog{\cdot}_\de$ turns the vector space $H_d(\R^{d+c},W,\R)$ into a finite dimensional Banach space, 
			\item 	$\no{k\be}^\R_\de= |k|\no{\be}^\R_\de$ for any $k\in\R$,
			\item and
			\begin{align}\label{eqfedb}
				\lim_{\substack{k\to\infty\\ k\in\N}}\frac{\no{k\be}^{\Z}_\de}{k}=\no{\be}^{\R}_\de.
			\end{align}
		\end{itemize}
	\end{fact} 
Now let us state the analogue of Lemma \ref{lemmsbd}. 
	\begin{lem}\label{lemmsbdb}
We have  \begin{align*}
		\sup_{[\Si]\in H_d(\R^{d+c},W,\Z)}		\frac{\no{[\Si]}_\de^{\Z}}{\no{[\Si]}_\de^{\R}}\le \upsilon(\de).
	\end{align*}
\end{lem}	\begin{proof}
By the first bullet of Lemma \ref{fcthomr}, $H_d(\R^{d+c},W,\Z)$ is free. The same proof of Lemma \ref{lemmsbd} works by replace homology of $M$ with relative homology of $(\R^{d+c},W)$. 	\end{proof}
	\subsection{Zhang's constructions}We need a way to construct metrics so that we can make a smoothly embedded representative of an integral homology class area-minimizing.
	\begin{lem}\label{lemzhang}
	Assume $N$ is a $d$-dimensional connected embedded compact closed submanifold of $M$. For any real number $\lam>0,$  there exists a smooth tubular neighborhood $U(N)$ of $N,$ a smooth Riemannian metric $h$ on $M$, and a smooth deformation retract $\pi_N$ of $U(N)$ onto $N$, such that
	\begin{itemize}
		\item any stationary varifold on $(M,h)$ whose support is not contained in $U(N)$ has area larger than $\lam \ms_h^{\Z}(N),$
		\item $\pi_N\du \dvol_{N}$ is a calibration form on $(U(N),h),$
		\item $N$ is the only integral current calibrated by $\pi_N\du \dvol_{N}$ in $(U(N),h)$,
		\item if $N$ represents a non-zero $\R$-homology class, then $\pi_N\du \dvol_{N}$ can be extended to a smooth calibration form on $(M,h),$
		\item $N$ is $\Z/n\Z$-area-minimizing in $U(N)$ for $n\in \Z_{\ge2}$.
	\end{itemize}
	\end{lem}
	\begin{proof}
	The lemma combines several results of Yongsheng Zhang's thesis \cite{YZt} and published in \cite{YZj,YZa}. Take a tubular neighborhood $U(N)$ of $N$ inside $M.$
The first bullet is \cite[Lemma 3.1]{YZj} applied to $U(N)$ and $M$. The second to the fourth bullet come from \cite[Sections 3.3, 3.4]{YZa}. The last bullet follows from \cite[Lemma 2.12]{ZLa}. For an alternative proof the last bullet, by \cite[Fact 7.0.6]{ZLns} $\pi\du_N\dvol_N$ is a calibration form on $U(N)$ implies that $\pi$ is a $d$-dimensional area-non-increasing map from $U(N)$ to $N$, i.e., \begin{align}\label{eqani}
		\ms^R(\pi\pf T)\le \ms^R(T),
	\end{align} for $R=\Z,$ or $\Z/n\Z$ with $n\in\Z_{\ge2 }$ and $T$ any $R$-flat chains supported in $U(N).$ For any $R$-flat chain $N'$ homologous to $N$ in $U(N)$, we necessarily have $\pi\pf N'=N$ since $\pi$ is a deformation retract. By (\ref{eqani}), we are done.
	\end{proof}

	\section{A priori density upper bounds under Assumption \ref{assumpm}}\label{secmon}
	In this section and the next section, we will carry out the first step in Section \ref{secsof}. bounds. Let us first prove a monotonicity formula, from which the desired density bounds follow directly.\subsection{Monotonicity formula across metrics in the compact setting}\label{ssmon}
In this subsection, we prove a monotonicity formula for stationary varifolds that holds across different metrics. Recall Assumption \ref{assumpm}. 
\begin{lem}\label{lemmon}
	There exist an open set $\Om_g$ containing $g$ in the space of smooth Riemannian metrics, positive real numbers $C_g,c_g,r_g>0$ depending on $g$,  such that, for any metric $h\in \Om_g$, any point $p\in M$ and  all radii $r\in(0,r_g],$ we have	\begin{itemize}
		\item for any $d$-dimensional integral stationary varifold $V$ on $M$ in the metric $h$, the function,
		\begin{align*}
			\exp(c_gr)V(\br(p,h))r^{-d},
		\end{align*} is monotonically increasing in $r$, where $V(\cdot)$ denotes the mass measure of $V$, and $\br(p,h)$ is the radius $r$ geodesic ball in metric $h$  centered at $p,$
		\item and the density $\ta_V({p,h})$ of $V$ at $p$ in $h$ is well-defined and satisfies
		\begin{align*}
			\ta_V({p,h})\le C_g V(M).
		\end{align*}
	\end{itemize}
\end{lem}
\begin{proof}
	The argument is the Riemannian adaptation of the first two sections of \cite{Dar}. First by \cite[Section 8 Theorem]{PEcir}, the injectivity radius $\inj_g(M)$ of $M$ in $g$ depends continuously on the metric $g$. Thus, there exists an open neighborhood $\Om_g$ containing $g$ in the space of Riemannian metrics so that 
	\begin{align*}
		r_g=\frac{1}{2}\inf_{h\in \Om_g}\inj_{h}(M)>0.
	\end{align*}	
	For any metric $h\in\Om_g,$ and any point $p\in M$, adopt a normal coordinate $(x_1,\cd,x_{d+c})$ in $\mathbf{B}_{r_g}(p,h)$. Let $\mathbf{r}_{p,h}$ denote the distance function to $p$ on $(M,h)$. By \cite[Chapter 2 ]{AG} we have 
	\begin{align}\label{rreq}
		\mathbf{r}_{p,h}\na \mathbf{r}_{p,h}=\sum_j x_j\pd_j.
	\end{align}
	A straightforward calculation using Taylor expansion shows that for any unit length tangent vector $u$ in $\mathbf{B}_{r_g}(p,h)$, we have
	\begin{align}\label{ineqr}
		|\ri{\na_u(\mathbf{r}_{p,h}\na \mathbf{r}_{p,h}),u}-1|\le c_{p,h},
	\end{align}for some $c_{p,h}>0$ which depends only on the $C^3$ norm of the ambient metric $h$, i.e.,
	\begin{align}\label{cnorm}
		c_{p,h}=O(\no{h}_{C^3}).
	\end{align}
	By compactness of $M$ and shrinking $\Om_g$ if necessary, we can set
	\begin{align*}
		c_g=\sup_{p\in M,h\in\Om_g}c_{p,h}<\infty.
	\end{align*}	
	Arguing as in the first two sections of \cite{Dar}, we deduce from (\ref{ineqr}) that $\exp(c_gr)V(\br(p,h))r^{-d}$ is monotonically increasing in $r$.
	
	The second bullet follows from the first bullet and \begin{align}
		\ta_V({p,h})=\lim_{r\to 0}\frac{V(\br(p,h))}{r^d}\le \exp(c_gr_g)V(\mathbf{B}_{r_g}(p,h))r_g^{-d}\le C_gV(M).
	\end{align}   
\end{proof}
\subsection{Density upper bound}
As a direct corollary, we deduce that
\begin{lem}\label{lemdst}
	Using the same notation as Lemma \ref{lemmon}, for $R=\Z$ or $R=\Z/n\Z$ with $n\ge 2$ and all $p\in M,h\in\Om_g$ and any $T\in\mn_h^{R}\left([\Si]\right),$ we have
	\begin{align*}
		\ta_T(p,h)\le C_{[\Si]},
	\end{align*}with $C_{[\Si]}>0$ depending only on $[\Si]$ and $\Om_g,$ not on $R.$
\end{lem}
\begin{proof}
	Use Lemma \ref{lemmon} and Fact \ref{fctncmp}.
\end{proof}
\section{A priori density bounds under Assumption \ref{assumpmb}}\label{secmonb}
In this section, we will continue carrying out the first step in Section \ref{secsof}. This this section will be devoted to the proof of the following lemma under Assumption \ref{assumpmb}.\begin{lem}\label{lemmondstb}
Assume Assumption \ref{assumpmb}.	Let $T\in\mn^R_\de([\Si])$ for $R=\Z$ or $R=\Z/n\Z$ with $n\in\Z_{\ge 2}.$ There exists $C_{[\Si]}>0$ depending only on $[\Si],$ not on $R,$ so that for any point $p\in\R^{d+c}$, then $\theta_T(p),$ the density of the $R$-mass measure of $T$ at $p,$ is at most $\theta(B),$ i.e.,
	\begin{align}\label{equfdst}
		\theta_T(p)\le C_{[\Si]}.
	\end{align} 
\end{lem}
Recall Lemma \ref{fcthomr}. It is convenient to fix a notation on the unique representative of $\pd[\Si]\in H_{d-1}(W,\Z).$
\begin{defn}
	Use $\pd\Si$ to denote the unique $\Z$-flat cycle representing the class  $\pd[\Si]\in H_{d-1}(W,\Z).$
\end{defn}
The case of smooth $W$ and subanalytic $W$ will be dealt with different methods. 
\subsection{Smooth $W$}
Let us start with the smooth case.
\begin{fact}\label{fctsmtw}
Assume that $c,d\in\Z_{\ge1},$ and $R=\Z$ or $R=\Z/n\Z$ with $n\ge 2$. 
If $W$ is a closed smooth $(d-1)$-dimensional submanifold of $\R^{d+c}$ and $B$ is a smooth $(d-1)$-dimensional integral cycle with $\supp B\s W,$ and $T$ is an $R$-area-minimizing flat chain with boundary $$\pd T=B,$$ then for any point $p\in\R^{d+c}$, we have $\theta_T(p)\le C_{B},$ with $C_{B}$ a positive real number depending only on $B.$
\end{fact}
If $W$ is smooth, then Lemma \ref{lemmondstb} follows from Fact \ref{fctsmtw} by setting $B=\pd\Si$. We want to emphasize that Fact \ref{fctsmtw} includes the case of $d=1$.
\begin{proof}
	Let $T\in\mn^R_\de([B]).$ We have
	\begin{align}\label{equmb}
		\ms^R_\de(T)=\no{[B]}^R\le\no{[B]}^{\Z}.
	\end{align} We will use both Allard's interior and Allard's boundary monotonicity formula for varifolds stationary outside $W$. The goal is to turn the uniform mass bound (\ref{equmb}) into the uniform density bound (\ref{equfdst})
	
First let us note that the mass measure $V_T$ of $T$ is a stationary integral varifold in $W\cp,$ since any diffeomorphism of $\R^{d+c}$ which is identity on $W$ must increase the mass. By Allard's boundary monotonicity formula \cite[(2) Theorem of Section 3.4]{WAb}, there exists $s_{[B]}>0,c_{[B]}>0$ depending only on $B$, so that for any $0<r\le s$, and any $p\in W$,
	\begin{align}\label{bdmon}
		\exp(c_{[B]}r)V_T(\mathbf{B}_r(p))r^{-d}
	\end{align}
	is monotonically increasing. We want to remark that though Allard commented \cite[Section 3.3]{WAb} that he assumes the dimension of the varifold is at least $2$, but the proof of \cite[(2) Theorem of Section 3.4]{WAb} did not use that assumption. Also, Allard did not assume that the boundary is connected. Thus, Allard's boundary monotonicity formula does apply to all cases we consider in this proof.	
	
	This implies that 
	\begin{align}\label{monpb}
		\ta_T(p)\le \exp(c_{[B]}s_{[B]})s_{[B]}^{-d}\ms^R(T)\le\exp(c_{[B]}s_{[B]})s_{[B]}^{-d}\no{[B]}^{\Z},
	\end{align}
	for any point $p\in  W.$

	On the other hand, for points $p$ at least $\frac{1}{2}s_{[B]}$ away from $ W$, by Allard's interior monotonicity formula in \cite[(2) Corollary of Section 5.1]{WA}, we have
	\begin{align}\label{monpi}
		\ta_T(p)\le \left(\frac{1}{2}s_{[B]}\right)^{-d}\ms^R(T)=2^ds_{[B]}^{-d}\no{[B]}^{\Z}.
	\end{align}
	For $p$ of distance at most $\frac{1}{2}s_{[B]}$ from $ W,$ let $q$ be a point on $ W$ minimizing $|p-q|.$ Then $\mathbf{B}_{|p-q|}(p)\s \mathbf{B}_{2|p-q|}(q),$ by Allard's interior monotonicity formula in \cite[(2) Corollary of Section 5.1]{WA}, we have
	\begin{align}
		\ta_T(p)\le& |p-q|^{-d}V_T(\mathbf{B}_{|p-q|}(p))\\\le& |p-q|^{-d}V_T(\mathbf{B}_{2|p-q|}(q))\label{monpib1}.
	\end{align}On the other hand, by Allard's boundary monotonicity formula \cite[(2) Theorem of Section 3.4]{WAb}, we have
	\begin{align}
		&|p-q|^{-d}V_T(\mathbf{B}_{2|p-q|}(q))\\=&2^d\exp(c_{[B]}(-2|p-q|))\exp(c_{[B]}(2|p-q|))(2|p-q|)^{-d}V_T(\mathbf{B}_{2|p-q|}(q))\\\le& 2^d\exp(c_{[B]}(s_{[B]}-2|p-q|))s_{[B]}^{-d}V(\mathbf{B}_{s_{[B]}}(q))\\\le&2^d\exp(c_{[B]}s_{[B]})s_{[B]}^{-d}\no{[B]}^{\Z}.\label{monpib2}
	\end{align}
	Bringing (\ref{monpib1}) into (\ref{monpib2}), we deduce that
	\begin{align}\label{monpib}
		\ta_T(p)\le 2^d\exp(c_{[B]}s_{[B]})s_{[B]}^{-d}\no{[B]}^{\Z}.
	\end{align}
	Combining all three cases, (\ref{monpb}) (\ref{monpi}) and (\ref{monpib}), we are done by setting $$C_B=2^d\exp(c_{[B]}s_{[B]})s_{[B]}^{-d}\no{[B]}^{\Z}.$$
\end{proof}	
\subsection{Subanalytic $W$}\label{secsubw}
For the case of subanalytic $W$, our goal is to prove the following lemma
\begin{fact}\label{fctsubw}Assume that $c,d\in\Z_{\ge1},$ and $R=\Z$ or $R=\Z/n\Z$ with $n\ge 2$. 
	If $W$ is a compact $(d-1)$-dimensional subanalytic set of $\R^{d+c}$ and $B$ is a $(d-1)$-dimensional integral cycle with $\supp B\s W,$ and $T$ is an $R$-area-minimizing flat chain with boundary $$\pd T=B,$$ then for any point $p\in\R^{d+c}$, we have $\theta_T(p)\le C_{B},$ with $C_{B}$ a positive real number depending only on $B.$
\end{fact}we need two ingredients
\begin{enumerate}
	\item use Gromov-White extended monotonicity formula to obtain a priori density upper bound by the mass of spherical projection of the boundary
	\item use Federer's spherical integral geometry formula to bound the mass of spherical projection of the boundary
\end{enumerate}
Let us start with the first ingredient.
\subsubsection{Gromov-White exterior density bound}
\begin{defn}
	For any point $p\in\R^{d+c}$, define $\pi_p$ to be the orthogonal projection of $\R^{d+c}\setminus p$ to the unit sphere $\mathbf{S}_1(p)$ centered at $p$, i.e.,
	\begin{align*}
		\pi_p(x)=p+\frac{1}{|x-p|}(x-p).
	\end{align*} 
\end{defn}
\begin{fact}\label{fctdsttb}
	For any point $p\in \R^{d+c},$ we have
	\begin{align*}
		\theta_T(p)\le \frac{\ms^R\left(\pi_p\left( B^R\right)\right)}{\hm^{d-1}(\mathbf{S}_1(0))}.
	\end{align*}
\end{fact}
Here $\pi_p\left( B^R\right)$ is the pushforward of $R$-flat chain $ B^R$ by $\pi_p$.
\begin{proof}
	Apply Gromov's \cite[Section 8.1-8.2]{MGfrm} extended monotonicity formula proved in \cite[Section 9.2, 9.3]{EWW}. We deduce that \cite[Theorem 1.3]{EWW}
	\begin{align*}
		\theta_T(p)\le \theta_{\operatorname{Cone}_p B^R}(p)
	\end{align*}Here $\operatorname{Cone}_p B^R$ denotes the cone over $\pd \Si$ centered at $p.$ Geometrically $\operatorname{Cone}_p B^R$ is formed by unions of rays starting from $p$ and passing through points in $\supp B^R$, counted with multiplicities in $R$ induced by $\pd \Si$. In other words, $\operatorname{Cone}_p B^R$ defined is by the pushforward of $ B^R\times[0,\infty)\s \R^{d+c}\times \R$ to $\R^{d+c}$ under the map
	\begin{align*}
		(x,y)\mapsto p+(x-p)y.
	\end{align*}
	By definition, $\operatorname{Cone}_p B^R$ is homogeneous and the slice of $\operatorname{Cone}_p B^R$ by the unit sphere centered at $p$ is $\pi_p\left( B^R\right)$, so we have
	\begin{align*}
		\theta_{\operatorname{Cone}_p B^R}(p)=\frac{\ms^R\left(\pi_p\left( B^R\right)\right)}{\hm^{d-1}(\mathbf{S}_1(0))}.
	\end{align*}
	We are done
\end{proof}
\subsubsection{Federer's integral geometry bound}
We will use integral geometry to bound $\ms^R\left(\pi_p\left( B^R\right)\right).$ Let us first define the density upper bound $|B^R|$
\begin{defn}\label{fctgrom}
	Use $\left|B^R\right|$ to denote the maximum density of $ B^R,$ i.e.,
	\begin{align*}
		\left|B^R\right|=\sup_{p\in\R^{d+c}}\theta_{ B^R}(p).
	\end{align*}
	When we are in the second case of Assumption \ref{assumpmb}, use $\left|[\Si]^R\right|$ to denote the maximum density of $\pd \Si^R,$ i.e.,
	\begin{align*}
		\left|[\Si]^R\right|=\left|\pd[\Si]^R\right|.
	\end{align*}
\end{defn}	
\begin{fact}
	$\left|B^R\right|$ is finite and depends only on $B^R.$\end{fact}When  we are in the second case of Assumption \ref{assumpmb}, we deduce that $\left|[\Si]^R\right|$ is finite and depends only on $\Si^R.$

\begin{proof}
	Apply \cite[Proposition 3.0.6]{GVvwl}, we deduce that $\sup_{p\in\R^{d+c}}\theta_W(p)$ is finite. The proposed fact then follows from Lemma \ref{fcthomr}.
\end{proof}	
\begin{lem}
	For any point $p\in \R^{d+c}$, we have 
	\begin{align}
		\label{eqmsr1}		\ms^R_\de\left(\pi_p\left( B^R\right)\right)=&\hm^{d-1}(\mathbf{S}^{d-1})\frac{1}{2}\int_{O(d+c)}\int_{\pi_p(W)\cap g.S^c}\theta_{\pi_p\left( B^R\right)}d\mathcal{H}^0 dg\\\le& \hm^{d-1}(\mathbf{S}^{d-1})\frac{1}{2}\mu_W\left|B^R\right|.\label{eqmsr2}
	\end{align}
\end{lem}	
The above lemma dates back to Gromov \cite[Section 8.1 Example (2)]{MGfrm}.
\begin{proof}
	Let us first prove (\ref{eqmsr1}).	Recall Fact \ref{fctvols}.	Fix a totally geodesic $c$-dimensional sphere $S^c$ in the unit sphere $\mathbf{S}^{d+c-1}(p)$ centered at $p$.	Apply the spherical integral geometry formula \cite[3.2.48 Theorem]{HF} with $A= \pi_p(W)$, $\ai=\theta_{\pi_p\left( B^R\right)}$, $B=S^c$, $\beta=1$, we deduce that
	\begin{align}
		\label{eqintg}\int_{O(d+c)}\int_{\pi_p(W)\cap g.S^c}\theta_{\pi_p\left( B^R\right)}(q)d\mathcal{H}^0(q) dg=\frac{\Gamma\left(\frac{d}{2}\right)\Ga\left(\frac{c+1}{2}\right)}{2\Ga\left(\frac{1}{2}\right)^{d+c+1}}\int_{\pi_p(W)}\theta_{\pi_p\left( B^R\right)}(q)d\mathcal{H}^{d-1}(q)\int_{S^c}d\mathcal{H}^{c}.
	\end{align}
	By definition of mass, we have	\begin{align}\label{eqintm}
		\ms^R_\de\left( B^R\right)=\int_{\pi_p(W)}\theta_{\pi_p\left( B^R\right)}(q)d\mathcal{H}^{d-1}(q).
	\end{align}
	On the other hand, the other factors on the right hand side of (\ref{eqintg}) evaluates to 	\begin{align}
		&\frac{\Gamma\left(\frac{d}{2}\right)\Ga\left(\frac{c+1}{2}\right)}{2\Ga\left(\frac{1}{2}\right)^{d+c+1}}(c+1)\frac{\Ga\left(\frac{1}{2}\right)^{c+1}}{\Ga\left(\frac{c+1}{2}+1\right)}
		=\Gamma\left(\frac{d}{2}\right)\frac{\Ga\left(\frac{1}{2}\right)^{c+1}}{\Ga\left(\frac{1}{2}\right)^{d+c+1}}\frac{c+1}{2}\frac{\Ga\left(\frac{c+1}{2}\right)}{\Ga\left(\frac{c+1}{2}+1\right)}\\
		=&\frac{\Gamma\left(\frac{d}{2}\right)}{\Ga\left(\frac{1}{2}\right)^{d}}=\frac{\Ga\left(\frac{(d-1)+1}{2}+1\right)/\frac{((d-1)+1)}{2}}{\Ga\left(\frac{1}{2}\right)^{(d-1)+1}}\\=&\frac{2}{\hm^{d-1}(\mathbf{S}^{d-1})}.\label{eqsvol}
	\end{align}
	Thus, bringing (\ref{eqsvol}) and (\ref{eqintm}) to (\ref{eqintg}), we deduce that (\ref{eqmsr1}) holds:
	\begin{align}\label{eqmst}
		\ms^R_\de(\pi_p\left( B^R\right))
		=&\hm^{d-1}(\mathbf{S}^{d-1})\frac{1}{2}\int_{O(d+c)}\int_{\pi_p(W)\cap g.S^c}\theta_{\pi_p\left( B^R\right)}(q)d\mathcal{H}^0(q) dg
	\end{align}	
	Now we go on to prove (\ref{eqmsr2}). By the area formula applied to mass, \cite[27.2 Remarks (3)]{LS}, we have
	\begin{align*}\label{eqdsts}
		\theta_{\pi_p\left( B^R\right)}(q)\le |\pi_p\m(q)\cap W^+|\left|[\Si]^R\right|,
	\end{align*}
	where $\pi_p\m(q)\cap W^+$ means the intersection of $\pi_p\m(q)$ with points in $W$ so that the Jacobian of $\pi_p$ is nonzero, i.e., $d\pi_p$ being full rank $(d-1)$. We deduce that 
	\begin{equation}\label{eqpti}
		\int_{\pi_p(W)\cap g.S^c}\theta_{\pi_p\left( B^R\right)}(q)d\mathcal{H}^0(q)\le\left|[\Si]^R\right|\int_{\pi_p(W)\cap g.S^c}|\pi_p\m(q)\cap W^+|d\mathcal{H}^0(q).	
	\end{equation}
	Let $P(g.S^c)$ denote the unique affine $(c+1)$-dimensional plane passing through $p$ and $g.S^c.$ We have
	\begin{align*}
		\pi_p(P(g.S^c)\cap W)=\pi_p(P(g.S^c))\cap \pi_p(W)=\pi_p(W)\cap g.S^c.
	\end{align*}
	This implies that
	\begin{align}\label{eqcnb}
		\int_{\pi_p(W)\cap g.S^c}
		|\pi_p\m(q)\cap W^+|d\mathcal{H}^0(q)=\#|	P(g.S^c)\cap W^+|
	\end{align}
	Bringing (\ref{eqcnb}) into (\ref{eqpti}), we deduce that 
	\begin{align}\label{eqdsf}
		\int_{\pi_p(W)\cap g.S^c}\theta_{\pi_p\left( B^R\right)}(q)d\mathcal{H}^0(q)\le\left|[\Si]^R\right|\#|	P(g.S^c)\cap W^+|
	\end{align}		
	Apply Lemma \ref{fctsub} by translating $p$ to the origin, we deduce that for Haar almost every $g\in O(d+c)$, we have
	\begin{align}\label{eqsubf}
		\#|	P(g.S^c)\cap W^+|\le\#|	g.P(S^c)\cap W|\le \mu_W.
	\end{align}
	Combining (\ref{eqsubf}) (\ref{eqdsf}) and (\ref{eqmst}), we deduce (\ref{eqmsr2}).
\end{proof}	
\subsubsection{Wrapping up the proof}
Combining Fact \ref{fctgrom} and Fact \ref{fctdsttb}, we deduce the following fact
\begin{fact}\label{fctsubts}
	For any point $p\in \R^{d+c}$, we have 
	\begin{align*}
		\theta_T(p)	\le\frac{1}{2}\mu_W\left|B^R\right|.
	\end{align*}
	When we are in the second case of Assumption \ref{assumpmb}, for any $T\in\mn^R_\de([\Si]^R)$ we have
	\begin{align*}
		\theta_T(p)	\le\frac{1}{2}\mu_W\left|[\Si]^R\right|.
	\end{align*}
\end{fact}Fact \ref{fctsubw} and thus Lemma \ref{lemmondstb} with $W$ subanalytic in Assumption \ref{assumpmb} follow directly. 
\subsection{$W$ a pair of planes}
We need a refinement of Lemma \ref{lemmondstb}  when we are dealing with a pair of planes. Recall Assumption \ref{assumpplane}.
\begin{defn}
	Define 
	\begin{align*}
		W_{P,Q}=\pd(P\res \mathbf{B}_1(0)), W_Q=\pd(Q\res \mathbf{B}_1(0))
	\end{align*} Set $ B_{P,Q}=W_P+W_Q$.
\end{defn}
We will also abuse notations and use $W_{P},W_{Q},W_{P,Q}$ to denote their supports, respectively, i.e., the intersections of $P,Q$ with the unit sphere $\mathbf{S}_1(0)$ and the union of the two intersections. 
\begin{fact}\label{fctplane}
	For $p\in \R^{d+c},$ and $T\in\mn^R_\de([\Si_{P,Q}])$ we have
	\begin{align*}
		\theta_T(p)\le 2
	\end{align*}and strict inequality holds for $p\not=0.$
\end{fact}
\begin{proof}
	By definition, we have $$\left|[\Si_{P,Q}]^R\right|=1.$$ Recall Fact \ref{fctsub} with $p$ translated to the origin. If an affine $(c+1)$-dimensional plane $V$ intersects $W_{P,Q}$ at finitely many points, then the cardinality $\#|V\cap W_{P,Q}|$ is at most $4$, i.e., $V$ intersecting $P,Q$ both at a line. Thus, we have $$\mu_{W_{P,Q}}=4.$$
	
	By Fact \ref{fctsubts}, we deduce that	\begin{align}\label{eqtt2}
		\theta_T(p)\le 2.\end{align} 
	
	To prove that strict inequality holds in (\ref{eqtt2}) for $p\not=0$, we prove the contrapositive.
	
	In other words, we will prove that equality in (\ref{eqtt2}) implies $p=0$.
	
	We need to recall the proof of Fact \ref{fctsubts}. Recall equality (\ref{eqmst}) Equality in (\ref{eqtt2}) holds implies that for Haar a.e.~ $g\in O(d+c)$, we have
	\begin{align}
		\label{eqdstp}	\int_{\pi_p(W_P)\cap g.S^c}\theta_{\pi_p\left(\pd \Si^R_{P,Q}\right)}d\mathcal{H}^0=2,\\
		\label{eqdstq}	\int_{\pi_p(W_Q)\cap g.S^c}\theta_{\pi_p\left(\pd \Si^R_{P,Q}\right)}d\mathcal{H}^0=2.
	\end{align}
	If $p\not\in P$, then take a flat coordinate system $(x_1,\cdots,x_{d+c})$ so that $P$ is spanned by $x_1,\cdots,x_d$-axes vectors and $p=(p_1,\cdots,p_{d+c})$ with $p_{d+c}>0.$ Consider the affine $(c+1)$-dimensional plane $P'_p$ parameterized by
	\begin{align*}
		P'_p=\{(p_1+x_1,\cdots,p_{c+1}+x_{c+1},p_{c+2},\cdots,x_{d+c})|(x_d,\cdots,p_{d+c})\in\R^{c+1}\}.
	\end{align*}
	Then $P'_p\cap P=\es.$ Thus, \begin{align*}
		\pi_p(P'_p)\cap \pi_p(W_P)=\es
	\end{align*} and there exists an open set $U(\pi_p(P'_p))$ containing $\pi_p(P'_p)$ so that
	\begin{align*}
		U(\pi_p(P'_p))\cap \pi_p(W_P)=\es.
	\end{align*}  However, $\pi_p(P'_p)$ is a totally geodesic $c$-dimensional sphere in $\mathbf{S}_1(p)$, so there exists $g_P$ so that  $g_P(S^c)=\pi_p(P'_p).$ For a neighborhood $U(g_P)$ of $g_P$ in $O(d+c)$ we have $h(\pi_p(P'_p))\s 	U(\pi_p(P'_p)).$ In other words, for $h\in U(g_P)$, we have $\pi_P(W_P)\cap h(S^c)=\es.$ 
	
	Thus, if (\ref{eqdstp}) holds for Haar a.e.~ $g\in O(d+c)$,  we must have $p\in V.$ The same argument works for $Q$ and we deduce that if (\ref{eqdstq}) holds then $p\in Q.$ We are done.
\end{proof}
\section{Regularity improvement and lifts to $\Z$}\label{secreg}
In this section we will carry out the second step outlined in Section \ref{secsof}.

First, let us unify Assumption \ref{assumpm} and Assumption \ref{assumpmb}.
\begin{assump}\label{assumpmc}
Consider the quadruples of configures $(M^{d+c},W,g,\Om_g)$ with $d,c\in\Z_{\ge 1}$. Assume that one of the followings holds
\begin{enumerate}
	\item (Assumption \ref{assumpm}) $M$ is a $(d+c)$-dimensional compact closed smooth manifold, $W$ is the empty set, $[\Si]\in H_d(M,W,\Z)$ is a $d$-dimensional integral homology class, and $g$ is a smooth metric, $\Om_g$ defined in Lemma \ref{lemmon},
\item (Assumption \ref{assumpmb}) $M=\R^{d+c}$ with $d\ge 2,$ $W$ is a $(d-1)$-dimensional compact closed smooth submanifold or a $(d-1)$-dimensional compact subanalytic  set, $[\Si]\in H_d(M,W,\Z)$  is a $d$-dimensional relative integral homology class, $g=\de$ is the standard flat metric and $\Om_g=\{g\}$.
\end{enumerate}
\end{assump}
Under Assumption \ref{assumpmc}, we will prove in this section that
\begin{fact}\label{fctfrmn}
Assume that $n\ge 3C_{[\Si]}$. For any metric $h\in \Om_g$, let $$T\in\mn^{\Z/n\Z}_h([\Si]).$$	Then any $\Z$-rectifiable current $S$ that is Federer's ``representative" modulo $n$ of $T$ satisfies 
\begin{align*}
	[\pd S]&\in H_{d-1}(W,\Z),\\\dim_{\mathcal{H}}\sing S\setminus W&\le d-2.
\end{align*}
In other words, $S$ is a relative cycle in $H_d(M,W,\Z)$ and the singular set of $S$ outside of $W $ is of Hausdorff dimension at most $(d-2).$
\end{fact}
Note that the above fact implies Corollary \ref{coras} for $M$ compact and $W=\es.$
\subsection{Proof of Fact \ref{fctfrmn}}\label{subsectionfed}Recall Lemma \ref{lemdst} and Lemma \ref{lemmondstb}. For any point $p\in M,$ by $n\ge 3C_{[\Si]}$, we have
\begin{align}\label{indst}
	\ta_T(p,h)<\frac{n}{2}.
\end{align}For any $\Z$-rectifiable current $S$ which is a Federer's ``representative" modulo $n$ of $T$, by Fact \ref{fctfcr}, $\pd S$ is a $\Z$-rectifiable current. By Lemma \ref{lemdv}, we have\begin{align}\label{eqsing}
	\spt \pd S\s\sing T.
\end{align}
By Lemma \ref{lemsingt} and Fact \ref{fctbd} applied to (\ref{eqsing}), we deduce that $\pd S$ is supported on $W.$ By Lemma \ref{fcthomr}, we are done.\section{Asymptotic analysis using Federer's theory of real flat chains}\label{secfed}In this section we will carry out the third step outlined in Section \ref{secsof}. Again assume Assumption \ref{assumpmc}. Recall the convention Definition \ref{defntsw}. We will prove the following fact
\begin{fact}\label{fctm}
Using the same notation as in Fact \ref{fctfrmn}, there exists a positive real number $N_{[\Si],g}$, so that for $n\ge N_{[\Si],g},$  we have $[S]\in[\Si]+nG_d(M,W,\Z)$. 	Furthermore, if $\tau_d|n,$ then $nG_d(M,W,\Z)=0$ and we have $[S]=[\Si].$
\end{fact}
	\subsection{Proof of Fact \ref{fctm}}\label{subsectiona} 
Recall Fact \ref{fctfrmn}. Since $S$ is a relative cycle, it is an integral current and belongs to a relative homology class $$[S]\in H_d(M,W,\Z).$$

Since $S$ is Federer's ``representative" modulo $n$ of $T$, we deduce that 
\begin{align}\label{eqsip}
	[S]^{\Z/n\Z}=[\Si]^{\Z/n\Z}.
\end{align}	
Let us first show that
\begin{fact}\label{fctsm}
	$S\in\mn_h^\Z([S])$ and $S\in \mn_h^{\Z}([\Si]+nH_d(M,W,\Z)).$
\end{fact}
\begin{proof}Let $Z$ be another integral current representing $[S].$ 	
	
	 By definition of Federer's ``representative" modulo $n$ (Fact \ref{fctrmn}), we have
	\begin{align}\label{eqn1}
		\ms_h^{\Z}(S)=\ms_h^{\Z/n\Z}(T).
	\end{align}On the other hand, by area-minimality of $T$ and that the canonical homomorphism $\Z\to \Z/n\Z$ is norm non-increasing, we deduce that
	\begin{align}\label{eqn2}
		\ms_h^{\Z/n\Z}(T)\le\ms_h^{\Z/n\Z}(Z^{\Z/n\Z})\le \ms_h^{\Z}(Z). 
	\end{align}
	Combining (\ref{eqn1}) and (\ref{eqn2}), we deduce that $S$ is $\Z$-area-minimizing in $[S].$ 
	
	By Fact \ref{fctrmn}, we have
	\begin{align}\label{eqssp}
		\no{[S]}_h^{\Z}=		\ms_h^{\Z}(S)=\ms_h^{\Z/n\Z}(T)=\no{[\Si]}_h^{\Z/n\Z}\le \inf_{[\sigma]\in [\Si]+nH(M,\Z)}\no{[\si]}_h^{\Z}.
	\end{align}
	As $[S]\in [\Si]+nH(M,W,\Z)$. We are done.
\end{proof}
Now, let us separate into two cases. First, when the $d$-th betti number of $M$ is zero, i.e., $b_d=0$, then $H_d(M,W,\Z)=G_d(M,W,\Z)$ and Fact \ref{fctsm} implies Fact \ref{fctm}.

From no on we assume that the $d$-th betti number $b_d$ is positive. Then we can use the asymptotic theory in developed in Section \ref{ssnorm} using Federer's theory for real flat chains.	Expand $[\Si'],[\Si]$ with respect to the basis in Assumption \ref{decompbs}, we have
\begin{align}\label{eqsp}
	[S]=&\sum_{j=1}^{b_d}\ga_jf_j+\sum_{i\in I}\eta_iv_i,\\
\label{eqsi}
	[\Si]=&\sum_{j=1}^{b_d}\ai_jf_j+\sum_{i\in I}\be_iv_i,
\end{align}
with each $\ai_j,\ga_j\in \Z,\be_i,\eta_i\in\Z\cap[0,p_j^{\nu_j}].$ 

From $[\Si]^{\Z/n\Z}=[S]^{\Z/n\Z},$ and Lemma \ref{lemnt}, we deduce that for $1\le j\le b_d,i\in I$ we have
\begin{align*}
	\ga_j=&\ai_j+k_jn,\\
	\eta_i=&\be_i+l_in,
\end{align*}with $k_j,l_i\in\Z.$
By Fact \ref{fctncmp}, Fact \ref{fctncmpb}, Fact \ref{fctfc}, second bullet of Lemma \ref{fcthomr}, Lemma \ref{lemnorm}, we have
\begin{align*}
	&\no{[S]}_h^{\Z}\ge\no{[S]}_h^{\R}
	=\no{\sum_{j=1}^{b_d}\ga_jf_j}^{\R}_h\\
	=&\no{\sum_{j=1}^{b_d}(\ai_j+k_jn)f_j}^{\R}_h\\
	\ge&(C(h))\m\sum_{j=1}^{b_d}\nog{(\ai_j+k_jn)f_j}_h\\
	\ge&(C(h))\m \sum_{j=1}^{b_d}\big(|k_j|n-|\ai_j|\big)\nog{f_j}_h.
\end{align*}where $C(h)$ is a continuous function from the space of Riemannian metrics to $\R_{>0}$. By (\ref{eqssp}), we deduce that
\begin{align}\label{eqsin}
	\no{[S]}_h^{\Z}\ge (C(h))\m \sum_{j=1}^{b_d}\big(|k_j|n-|\ai_j|\big)\nog{f_j}_h.
\end{align}
By shrinking $\Om_g$ if necessary to a smaller open set $\Om_{[\Si],g}$, we can assume that
\begin{align*}
	\inf_{h\in \Om_g}C(h)\m>&0,\\\inf_{h\in\Om_g}\min_{1\le j\le b_d}\nog{f_j}_h>&0,\\
	\sup_{h\in \Om_g}\no{[\Si]}^{\Z}_h<&\infty.
\end{align*}
\begin{defn}\label{defnnsg}
	Set 
	\begin{align*}
		N_{[\Si],g}=3\max\left\{C_{[\Si]},\frac{	\sup_{h\in \Om_g}\no{[\Si]}^{\Z}_h}{\left(\inf_{h\in \Om_g}(C(h))\m\right)\left(\inf_{h\in\Om_g}\min_{1\le j\le b_d}\nog{f_j}_h\right)}+\max_{1\le j\le b_d}|\ai_j|\right\}.
	\end{align*}
\end{defn}
Apply Definition \ref{defnnsg} for $n\ge N_{[\Si],g}$. If for some $j$ we have $k_j\not=0,$ then
\begin{align*}
	&(C(h))\m \sum_{l=1}^{b_d}\big(|k_l|n-|\ai_l|\big)\nog{f_l}_h\\	\ge &3\frac{(C(h))\m\nog{f_j}_h	}{\left(\inf_{h\in \Om_g}(C(h))\m\right)\left(\inf_{h\in\Om_g}\min_{1\le l\le b_d}\nog{f_l}_h\right)}\sup_{h\in \Om_g}\no{[\Si]}^{\Z}_h\\&+2(C(h))\m\nog{f_j}_h	\max_{1\le l\le b_d}|\ai_l|\\
	\ge&3 \no{[\Si]}^{\Z}_h,
\end{align*} which is impossible by (\ref{eqsin}). We deduce that $\sum_{j=1}^{b_d}|k_j|=0.$ We are done.
\section{Proof of Theorems \ref{thmv}, \ref{thmvg}, \ref{thmvb}, \ref{thmcpm}, and Corollary \ref{coras}}\label{secpm}
	\subsection{Proof of Theorems \ref{thmv}, \ref{thmvg}, \ref{thmvb}}
Let us first prove Theorem \ref{thmv} and \ref{thmvb}.	The first bullet of Theorem \ref{thmv} follows from Fact \ref{fctm}. To prove the second bullet of Theorem \ref{thmv} from Fact \ref{fctm}, it suffices to show that
\begin{fact}\label{lemssp}
	Using the same notation as in Fact \ref{fctrmn},  $S$ as a Federer's ``representative" modulo $n$ is unique. In other words, if $S'$ is another Federer's ``representative" modulo $n$ of $T$, then we have	$S=S'$.
\end{fact}
	\begin{proof}		
First,  $S^{\Z/n\Z}=(S')^{\Z/n\Z}=T$ implies $(S'-S)^{\Z/n\Z}=0$.  As $S$ and $S'$ has the same support and they both have boundaries supported in $W$, they must have the same regular set $$\reg S=\reg S'=\reg S\cap \reg S'.$$ In each connected component of $\reg S\cap \reg S'$, $S'-S$ equals a smooth submanifold counted with multiplicity divisible by $n.$  by (\ref{indst}), $S$ and $S'$ has density strictly less than $\frac{n}{2}$ Hausdorff $d$-dimensional a.e.~. Thus, $S-S'$ has density strictly less than $n$ Hausdorff $d$-dimensional a.e.~. 	This implies that $S$ and $S'$ coincides on their regular set. Since $\sing S\setminus W=\sing S'\setminus W=\sing $ has Hausdorff dimension at most $(d-2)$, we deduce that $S=S'.$\end{proof}
Theorem \ref{thmvg} follows directly from Fact \ref{fctsm}, Fact \ref{fctm} and Lemma \ref{lemssp}. 
	\subsection{Proof of Theorem \ref{thmcpm}}Let us show how Theorem  \ref{thmcpm} reduces to Lemma \ref{lemcpm}. 
	
	The argument in	\cite[Section 5]{GLac} shows that it suffices to consider pairs of planes that only intersect at the origin. For $d\le c,$ \cite[Lemma 7.5]{HL} shows that a pair of $d$-dimensional planes intersecting in $\R^{d+c}$ only at the origin must be contained in a Euclidean subspace of dimension $2d$.
	
	Thus, we only need to prove Lemma \ref{lemcpm}.
	\subsubsection{Proof of Lemma \ref{lemcpm}}\label{secpfred}
By Fact \ref{fctncmpb}, necessity is clear. We need to deal with sufficiency. In the rest of this subsection, we will prove that if $P+Q$ is $\Z$-area-minimizing, then $P^{\Z/n\Z}+Q^{\Z/n\Z}$ is $\Z/n\Z$-area-minimizing  with $n\ge 4.$

Again we use the three steps outlined in Section \ref{secsof}.
\subsubsection{First step and second step: density bound and lift to integral currents}
Let $T\in\mn_\de^{\Z/n\Z}(W_P+W_Q)$. By Fact \ref{fctplane}, we deduce that $\theta_p<2\le \frac{n}{2}$ for any point $p\in\R^{d+c}$ unless $p=0.$ Apply Lemma \ref{lemsingt}, we deduce that $$\dim\sing T\cap W\cp=\max\{d-2,\dim\{p\}\}=d-2.$$ By Fact \ref{fctfcr} and (\ref{eqsing}), we deduce that $\pd S$ restricted to $\sing T$ is a $(d-1)$-dimensional $\Z$-flat chain supported on a set of Hausdorff dimension at most $(d-2)$. Thus, we deduce that $\pd S$ must be supported on $W.$ We are done. In other words,
\begin{fact}\label{fctfrmnp}
	Any $\Z$-rectifiable current $S$ that is Federer's ``representative" modulo $n$ of $T$ determines a homology class $\pd S\in H_{d-1}(W_{P,Q},\Z)$ provided $n\ge 4$.
\end{fact}
\subsubsection{Third step: determining the homology class}
For this step we use a direct calibration bound instead of relying on Federer's asymptotics. By Lemma \ref{fcthomr}, we deduce that
\begin{align*}
	\pd S=aW_P+bW_Q\textnormal{ where }a,b\equiv 1\mod n.
\end{align*}By \cite[Remarks right after 2.2 Theorem]{FMcalv}, we have
\begin{align*}
	\no{P\res\mathbf{B}_1^{d}(0)}^{\Z/n\Z}_\de=\no{Q\res\mathbf{B}_1^{d}(0)}^{\Z/n\Z}_\de=\mathcal{H}^d(\mathbf{B}_1^{d}),
\end{align*}
and the unique $\Z/n\Z$-area-minimizing flat chain spanned by $W_P$ and $W_Q$ are $P\res \mathbf{B}_1(0)$ and $Q\res \mathbf{B}_1(0)$.

Our goal is to prove that $a=b=1.$ First, we have
\begin{align}\label{eqbbdu}
	\ms^\Z_\de(S)=\ms^{\Z/n\Z}_\de(T)\le\no{P\res\mathbf{B}_1^{d}(0)}^{\Z/n\Z}_\de+\no{Q\res\mathbf{B}_1^{d}(0)}^{\Z/n\Z}_\de=2\mathcal{H}^d(\mathbf{B}_1^{d}).
\end{align}
If $a,b$ are of the same sign, then by \cite{GLac,DN}, there is a constant $d$-form calibrating both $P$ and $Q.$ Thus, we have
\begin{align*}
	\ms^\Z(S)=&\no{aP\res\mathbf{B}_1^{d}(0)+bQ\res\mathbf{B}_1^{d}(0)}^{\Z}_\de\\=&\no{aP\res\mathbf{B}_1^{d}(0)+bQ\res\mathbf{B}_1^{d}(0)}^{\R}_\de\\=&|a|\no{P\res\mathbf{B}_1^{d}(0)}^{\R}_\de+|b|\no{P\res\mathbf{B}_1^{d}(0)}^\R_\de\\=&(|a|+|b|)\mathcal{H}^d(\mathbf{B}_1^{d}).
\end{align*}
By (\ref{eqbbdu}) we deduce that
\begin{align*}
	|a|+|b|\le 2.
\end{align*}As $n\ge 4$ and $a,b\equiv 1\mod n$ we deduce that $a=b=1$ when $a,b$ are of the same sign.

To finish the proof, we need to show that 
\begin{fact}\label{fctabimp}
 $a,b$ cannot be of of opposite signs.
\end{fact}Let us suppose the opposite. Assume that $a>0,b<0$. Set
\begin{align*}
	a=1+na',b=1-nb',
\end{align*}with $a'\in\Z_{\ge0},b'\in\Z_{\ge 1}$. Let $P^\ast$ and $Q^\ast$ denotes the unit simple $d$-form dual to the planes $P$ and $Q$, respectively. By \cite[Section II Proposition 7.10]{HL}
\begin{fact}\label{fctpqcal}
	The forms $P^\ast$ and $Q^\ast$ are calibration forms calibrating only $P$ and $Q$, respectively.
\end{fact}
We have
\begin{align}
		\ms^{\Z}(S)\ge S\left(P^\ast\right)\label{eqsp0}.\end{align}
By Fact \ref{fctpqcal},	equality in (\ref{eqsp0}) holds if $P\du$ calibrates $S.$

On the other hand, as $\pd S=aW_p+bW_q,$ by Stokes theorem, we deduce that
\begin{align}
	S\left( P^\ast\right)=a(P\res \mathbf{B}_1(0))\left( P^\ast\right)+b(Q\res \mathbf{B}_1(0))\left( P^\ast\right).
\end{align}
Using $a>0,b<0$ and $(Q\res \mathbf{B}_1(0))\left( P^\ast\right)\le\ms_\de^{\Z}(Q\res\mathbf{B}_1(0))$, we deduce that
\begin{align}
	a(P\res \mathbf{B}_1(0))\left( P^\ast\right)+b(Q\res \mathbf{B}_1(0))\left( P^\ast\right)\ge ((1+na')-(nb'-1))\mathcal{H}^d(\mathbf{B}_1^{d}).
	\label{eqsp},
\end{align} where equality holds if and only if $P^\ast$ calibrates $Q$, which is impossible. Bringing the chain of inequalities from (\ref{eqsp}) back to (\ref{eqsp0}) we deduce that
\begin{align}
\label{ineq1}	\ms^{\Z}_\de(S)>((1+na')-(nb'-1))\mathcal{H}^d(\mathbf{B}_1^{d}).
\end{align}
Similarly we have the chain of inequalities
	\begin{align*}	\ms^{\Z}(S)\ge S\left(Q^\ast\right)=a(P\res \mathbf{B}_1(0))\left(Q^\ast\right)+b(Q\res \mathbf{B}_1(0))\left(Q^\ast\right)\ge ((nb'-1)-(1+na'))\mathcal{H}^d(\mathbf{B}_1^{d}),
	\end{align*}where equalities can hold only if $Q\du$ calibrates $P,$ which is impossible. Thus, we have
	\begin{align}\label{ineq2}
		\ms^{\Z}(S)> ((nb'-1)-(1+na'))\mathcal{H}^d(\mathbf{B}_1^{d}).
	\end{align}
Bringing (\ref{ineq1}), (\ref{ineq2}) and (\ref{eqbbdu}) together, we deduce that
\begin{align}\label{eqab}
	2> |(1+na')-(nb'-1)|=|2+n(a'-b')|.
\end{align}
Since $n\ge 4,$ inequality (\ref{eqab}) can never hold, a contradiction. Thus, we cannot have $a>0,b<0$. Assume $a<0,b>0$, we can repeat the argument by exchanging the notations $P$ and $Q$ and also arrive at a contradiction. This finishes the proof of Fact \ref{fctabimp} and we are done.	
\section{The case of $d=1$ and alternative proof of Theorem \ref{thmcm}}\label{secd=1}
Theorem \ref{thmcm} follows directly from from Theorem \ref{thmvb} for $d\ge 2.$ In this section, we will prove the remaining $d=1$ case of Theorem \ref{thmcm}. Actually we will provide an alternative proof of Theorem \ref{thmcm} that circumvents Federer's theory of real flat chains.

Let us recall Assumption \ref{assumpmb}. Let $B$ be an integral current cycle supported in $W$. Define
\begin{defn}
	Let $T_n$ be a $\Z/n\Z$-area-minimizing flat chain with $\pd T=B^{\Z/n\Z}.$
\end{defn}We need to prove that 
\begin{fact}
There exists $N_B$ depending only on $B$ so that	$T_n$ is an integral current with $\pd T_n=B$ for $n\ge N_B.$
\end{fact}  Recall Fact \ref{fctsmtw} and Fact \ref{fctsubw}. For any point $p\in\R^{d+c},$ we have
\begin{align*}
	\theta_{T_n}(p)\le C_{B}.
\end{align*}
From now on assume $n\ge 3C_B$, By Lemma \ref{lemsingt} and Fact \ref{fctbd}, we deduce that $T_n$ admits a Federer's ``representative" modulo $n$, $S_n,$ so that
\begin{align*}
	\supp\pd S_n\subset W.
\end{align*}
By Lemma \ref{fcthomr} and Lemma \ref{lemnt}, $\pd S_n$ can be represented by a finite simplicial chain so that
\begin{align*}
\pd	S_n=B+nV
\end{align*} for some $(d-1)$-dimensional integral cycle $V$ with $[V]\in H_{d-1}(W,\Z).$

Let $p$ be a regular point of $V$ which must also be a regular point of $B.$ In a neighborhood $U(p)$ of $p,$ we have
\begin{align*}
	(B+nV)\res U(p)=|\theta_B(p)\pm n\theta_V(p)|W\res U(p),
\end{align*} where the sign depends on whether $B,V$ are oppositely oriented or not.

Apply the boundary case of Almgren's stratification of $\Z$-area-minimizing currents \cite{BWst}, we deduce that for Hausdorff $(d-1)$-dimensional almost every $q\in\supp \pd S\cap U(p),$ there is a tangent cone to
$B$ that splits as a product
\begin{align*}
	C\times \R^{d-1},
\end{align*}
where $C$ is a $1$-dimensional cone $\Z$-area-minimizing cone with
\begin{align}\label{eqbdc}
	\pd C=|\theta_B(p)\pm n\theta_V(p)|\{0\}
\end{align}

Set $$N_B=3C_{B}+|B|,$$ where $|B|$ is the maximum density of $B$ defined in Definition \ref{fctgrom}. If $n\ge N_B,$ then
\begin{align*}
|\theta_B(p)\pm n\theta_V(p)|\ge n\theta_n(V)-\theta_B(p)\ge 3C_{B}.
\end{align*}
Since $C$ is a union of rays starting from the origin and with orientations, $C$ must consists of at least $|\theta_B(p)\pm n\theta_V(p)|$ rays, multiplicity counted, in order to have (\ref{eqbdc}). This implies that $C$ has density at least $\frac{3}{2}C_B$. By \cite[Section 3.3]{NWrs}, $C$ and $C\times \R^{d-1}$ have the same desnity, which contradicts Fact \ref{fctsmtw} and Fact \ref{fctsubw}. Thus, $V=0$ and we are done.
\section{Proof of  Theorem \ref{thmvn} and Theorem \ref{thmnc}}\label{secpfs}
	In this section, we will prove Theorem \ref{thmvn} and Theorem \ref{thmnc}. Let us start with Theorem \ref{thmnc}, which is easier.
	\subsection{Proof of Theorem \ref{thmnc}}
	Let us start with the following fact
	\begin{fact}
		A $\Z$-area-minimizing current in a non-zero torsion homology class can never be calibrated. 
	\end{fact}
	This follows from the definition of $\R$-homology, as any closed $\R$-flat cochain $\ai$ must evaluate to $0$ on a torsion class.
	
	Thus, to prove Theorem \ref{thmnc}, we only need to consider non-torsion classes. Theorem \ref{thmnc} follows directly from the following lemma.
	\begin{lem}\label{lemcal}
		If $d\ge 1,c\ge 2,$ $[\Si]$ is a non-torsion class and $M$ is an orientable, then for any $\lam>0$, there exists a non-empty open set $\Om'_{[\Si],\lam}$ in the space of Riemannian metrics, depending on $[\Si]$ and $\lam,$ such that for any metric $h\in \Om'_{[\Si],\lam},$ we have
		\begin{align}
			\no{[\Si]}_h^{\Z}>\lam\mu_{[\Si]}\nog{[\Si]}_h.\label{inzr}
		\end{align} 
	\end{lem}
	We emphasize that by Definition \ref{defnhom}, we have $\mu_{[\Si]}\ge 2.$
	\begin{proof}
		By Fact \ref{fctnr}, it suffices to show the existence of one metric $h$ in which inequality (\ref{inzr}) holds.
		
		Let $N$ be a smoothly embedded connected submanifold representing $\mu_{[\Si]}[\Si].$ Apply Lemma \ref{lemzhang} to $N$ to obtain a metric $h$ on $M,$ tubular neighborhood $U(N)$, and the deformation retract $\pi_N.$
		
		By the last bullet of Lemma \ref{lemzhang} and Fact \ref{fcthl}, we have
		\begin{align}\label{eqzrn}
			\no{\mu_{[\Si]}[\Si]}_h^{\Z}=\no{\mu_{[\Si]}[\Si]}_h^{\R}=\ms_h^{\Z}(N).
		\end{align}
		By the first bullet of Fact \ref{fctnrt}, we have
		\begin{align}\label{eqsmn1}
			\no{[\Si]}_h^{\R}=\frac{1}{\mu_{[\Si]}}\ms_h^{\Z}(N).
		\end{align}
		Let $T\in\mn_h^{\Z}\left([\Si]\right).$ Since $\pi$ is a deformation retract of $U(N)$ onto $N,$ by homotopy invariance of homology \cite{AH}, we deduce that $H_d(U(N),\Z)=\Z[N]=\Z[\mu_{[\Si]}\Si].$ Thus, any integral current whose support is contained in $U(N)$ is homologous to some integer multiple of $\mu_{[\Si]}[\Si].$ Since $[\Si]$ is a non-torsion class, we deduce that
		\begin{fact}
			The support of $T$ cannot be contained in $U(N)$.
		\end{fact} 
		By the first bullet of Lemma \ref{lemzhang}, we deduce that
		\begin{align}\label{insmn}
			\no{[\Si]}_h^{\Z}>\lam\ms_h^{\Z}(N).
		\end{align}
		Combining (\ref{eqsmn1}) and (\ref{insmn}). We are done
	\end{proof}
	\subsection{Proof of Theorem \ref{thmvn}}
	The proof of Theorem \ref{thmvn} is reliant on the proof of Lemma \ref{lemcal}. 
	
	By the last bullet of Lemma \ref{lemzhang}, $N$ is $\Z/n\Z$-area-minimizing in $U(N)$. By the first bullet of Lemma \ref{lemzhang}, we deduce that 
	\begin{align*}
		\no{\mu_{[\Si]}[\Si]}_{g}^{\Z/n\Z}=\ms_h^{\Z}(N).
	\end{align*}
	If $n\ge 2$ and $\operatorname{gcd}\left(n,\mu_{[\Si]}\right)=1,$ then $\mu_{[\Si]}$ is invertible in $\Z/n\Z.$ By Lemma \ref{lemhn}, we have
	\begin{align}\label{insn}
		\frac{2}{n}\no{[\Si]}_{g}^{\Z/n\Z}\le	\no{\mu_{[\Si]}[\Si]}_{g}^{\Z/n\Z}=\ms_h^{\Z}(N).
	\end{align}
	Combining (\ref{insn}) and (\ref{insmn}), we have
	\begin{align*}
		\no{[\Si]}_h^{\Z}>\frac{2\lam}{n}\no{[\Si]}_h^{\Z/n\Z}.
	\end{align*}
	By Fact \ref{fctnr}, we are done by setting $\lam=\frac{n\rh}{2}$ and $\Om'_{[\Si],n,\rh}=\Om'_{[\Si],\lam}$ using the same notations as in Lemma \ref{lemcal}.
		\section{Proof of Theorem \ref{thmkun}}\label{secpfp}	Let us first prove that,
	\begin{lem}\label{lemfaa}
		For any metric $h$ on $M,$ there exists a non-empty open set $\Om_{[\Si],h}\ni h$ in the space of Riemannian metrics on $M$, and an integer $\nu_{[\Si],h}$ such that for any integer $\nu\ge \nu_{[\Si],h}$ and any metric ${\hat{h}}\in \Om_{[\Si],h}$, we have
		\begin{align}
			1\le\frac{\no{\nu[\Si]}^{\Z}_{{\hat{h}}}}{\no{\nu[\Si]}^{\R}_{{\hat{h}}}}\le\frac{3}{2}.
		\end{align} 
	\end{lem}
	\begin{proof}
		By (\ref{eqfed}), there exists an integer $ \nu_{[\Si],h}$ such that for any $\nu\ge \nu_{[\Si],h}$, we have 
		\begin{align}\label{msbdg}
			1\le\frac{\no{\nu[\Si]}^{\Z}_h}{\no{\nu[\Si]}^{\R}_h}\le\frac{9}{8}.
		\end{align} There also exists an open set $\Om_{[\Si],h}\ni h$ in the space of Riemannian metrics on $M$, such that for any metric ${\hat{h}}\ge \Om_{[\Si],h}$, we  have \begin{align}
			\left(\frac{8}{9}\right)^{\frac{2}{d}}{\hat{h}}\le h\le\left(\frac{9}{8}\right)^{\frac{2}{d}}{\hat{h}},
		\end{align} as quadratic forms, which by definition of mass \cite[p.358]{HF}, implies that
		\begin{align}\label{mtbdh}
			\frac{8}{9}\ms_{{\hat{h}}}^{R}\le \ms_h^{R}\le	\frac{9}{8}\ms_{{\hat{h}}}^{R},
		\end{align}for $R=\Z$ or $\R$. 
		Combining (\ref{mtbdh}) and (\ref{msbdg}), by Fact \ref{fctmca}, we have
		\begin{align*}
			\frac{\no{\nu[\Si]}^{\Z}_{{\hat{h}}}}{\no{\nu[\Si]}^{\R}_{{\hat{h}}}}\le \frac{\frac{9}{8}\no{\nu[\Si]}^{\Z}_h}{\frac{8}{9}\no{\nu[\Si]}^{\R}_h}
			\le\left(\frac{9}{8}\right)^{3}.
		\end{align*} Since $\left(\frac{9}{8}\right)^3<\frac{3}{2}$, we are done.
	\end{proof}
	Let us first find the non-empty open set of metrics on $M'$.	Apply Lemma \ref{lemmsbd}. There exists a non-empty open set $\Om_{[\Si],g',\ka}'\ni g'$ in the space of Riemannian metrics on $M'$, such that for any $h'\in \Om_{[\Si],g',\ka}'$, and any non-torsion class $[\si']$ of any dimension $0\le m\le d'+c'$ on $M',$ we have,
	\begin{align*}
		\no{[\si']}_{h'}^{\R}\le \no{[\si']}_{h'}^{\Z}\le v({h'})	\no{[\si']}_{h'}^{\R},	
	\end{align*}for some real number $v({h'})>0$ depending continuously on $h'$.	
	
	By shrinking $\Om'_{[\Si],g',\ka}$ if necessary, we can assume that 
	\begin{align*}
		\sup_{h'\in \Om'_{[\Si],g',\ka}}v(h')=v_{g'}<\infty,
	\end{align*}
	and thus for all $h'\in\Om'_{[\Si],g',\ka},$ we have
		\begin{align}\label{eqrv}
		\no{[\si']}_{h'}^{\R}\le \no{[\si']}_{h'}^{\Z}\le v_{g
		}\no{[\si']}_{h'}^{\R},	
	\end{align}
	
	Next let us find the non-empty open set of metrics on $M$.	Apply Lemma \ref{lemcal}. For any $\lam>0$, there exists a non-empty open set $\Om_{[\Si],\lam}$ in the space of Riemannian metrics on $M$ such that for any metric $\check{h}$ in $ \Om_{[\Si],\lam}$, we have $			\no{[\Si]}_{\check{h}}^{\Z}>\lam\mu_{[\Si]}\nog{[\Si]}_{\check{h}}.$ 
	
	Fix one metric $\check{h}$ in $\Om_{[\Si],\lam}$. Apply Lemma \ref{lemfaa} to obtain a non-empty open subset $\Om_{[\Si],\lam,\check{h}}\ni \check{h}$ contained in $ \Om_{[\Si],\lam}$, and an integer $\nu_{[\Si],\check{h}}$, such that for any metric $h\in\Om_{[\Si],\lam,\check{h}}$, and any integer $\nu\ge \nu_{[\Si],\check{h}}$, we always have,
	\begin{align}\label{eqn3}
		\no{\nu[\Si]}_h^{\R}\le \no{\nu[\Si]}^{\Z}_h\le \frac{3}{2}\no{\nu[\Si]}_h^{\R},
	\end{align}
	and 
	\begin{align}\label{eqlm}
		\no{[\Si]}^{\Z}_h> \lam\mu_{[\Si]}\no{[\Si]}_h^{\R}.
	\end{align}
	
	For any integer $\nu\ge \nu_{[\Si],\check{h}}$, if $[\si']=\nu[\si'']\in \nu H_m(M',\Z)$ is a non-torsion class. 
	In these cases, using (\ref{ineqprod}), (\ref{eqrv}), and (\ref{eqn3}), we have
	\begin{align}\label{eqf1}
		&\no{[\Si]\times [\si']}^{\Z}_h=\no{[\Si]\times \nu[\si'']}_h^{\Z}\\=&\no{\nu[\Si]\times [\si'']}_h^{\Z}\\\le&\no{\nu[\Si]}_h^{\Z}\times \no{ [\si'']}_h^{\Z}\\\le&\frac{3}{2}\no{\nu[\Si]}_h^{\R}\times v_{g'}\no{ [\si'']}_h^{\R}\\
		=&\frac{3}{2}v_{g'}\nu \no{[\Si]}_h^{\R}\no{ [\si'']}_h^{\R}.
	\end{align}
	On the other hand, using (\ref{eqlm}), we have
	\begin{align}\label{eqf2}
		&	\no{[\Si]}^{\Z}_h\times \no{[\si']}^{\Z}_h\\>& \lam\mu_{[\Si]}\no{[\Si]}_h^{\R}\times \no{[\si']}^{\R}_h\\
		=&	\lam\mu_{[\Si]}\nu\no{[\Si]}_h^{\R} \no{[\si'']}^{\R}_h.
	\end{align}
	Combining the series of inequalities starting from (\ref{eqf1}) and (\ref{eqf2}), we have
	\begin{align*}
		\no{[\Si]}^{\Z}_h\times \no{[\si']}^{\Z}_h>	\lam\frac{2\mu_{[\Si]}}{3v_{g'}}\no{[\Si]\times [\si']}^{\Z}_h.
	\end{align*}	Recall from Definition \ref{defnhom} that $\mu_{[\Si]}$ is a topological invariant independent of metrics, while $v_{g'}$ defined in (\ref{eqrv}) depends only on $g'.$ Thus by setting $\lam=\frac{3\nu_{g'}}{2\mu_{[\Si]}}\ka$ and $\Omega_{[\Si],g',\ka}=\Omega_{[\Si],\lam}$,	we are done.
\section{Remarks about the main theorems}\label{secrem}
	\subsection{Other geometric variational problems}
	Many geometric variational problems can also be formulated on chain spaces over different coefficients, e.g., Almgren-Pitts min-max \cite{JPmm,MNtsc}, mean curvature flow \cite{CMm1,BWcfm2}, Song's spherical Plateau problem \cite{ASsp}, minimizers of elliptic \cite{SSA} and other general functionals, etc. Motivated by these instances, we propose the following analogue of the Hasse principle:
	\begin{conj}\label{conj}(The Hasse principle for variational problems)
		For a variational problem that can be formulated on chain spaces over $\R,\Z$ and $\Z/n\Z$ for $n\ge 2,$ the solutions in $\Z$-chains can be reconstructed from the corresponding solutions in $\R$-chains and $\Z/n\Z$-chains for all $n \geq 2$.
	\end{conj} 
	For instance, an obvious direction is that of classical Plateau's problems on Euclidean spaces. 
	
	It is worth emphasizing that in number theory, a great deal of research revolves around obstructions to the Hasse principle. For example, the Tate–Shafarevich group is heuristically interpreted as the obstruction to the Hasse principle \cite[Sections 7, 8]{BMlgn}, and it plays a central role in many of the field’s deepest conjectures, such as the Birch and Swinnerton–Dyer conjecture. This naturally leads us to ask:
	\begin{conj}
		What are the obstructions to the Hasse principle for variational problems?
	\end{conj}
	\subsection{Why real homology is needed in the compact setting?}\label{secdrh}
	In the statement of Theorem \ref{thmv}, there is no appearance of $H_d(M,\R).$ The reader might wonder why we still say the theorem manifests the Hasse principle, which needs both real solutions and solutions modulo prime powers. 
	
	Information from $H_d(M,\R)$ in Section \ref{subsectiona} is used to give lower bounds on the growth of $\no{\cdot}_g^{\Z}$ on $H_d(M,\Z).$
	Heuristically speaking, the Banach space $H_d(M,\R)$ with norm $\no{\cdot}_g^{\R}$ represents the asymptotic behavior of $H_d(M,\Z)$ with norm $\no{\cdot}_g^{\Z}$ at infinity. 
	
	Recall Fact (\ref{eqfed}):
	\begin{align*}
		\lim_{\substack{k\to\infty\\ k\in\N}}\frac{\no{k[\Si]}_g^{\Z}}{k}=\no{[\Si]}_g^{\R}.
	\end{align*}In sharp contrast, by (\ref{eqzrn}) and (\ref{insmn}), we have constructed metrics $h$ with $\no{[\Si]}_h^{\Z}>\lam\no{\mu_{[\Si]}[\Si]}_h^{\Z}$ for any $\lam>1.$ The only condition we used about $\mu_{[\Si]}$ is that $\mu_{[\Si]}[\Si]$ admits a smooth representative. By using Lemma \ref{lemthom}, our proof of Lemma \ref{lemcal} can be modified to show that
	\begin{cor}\label{corbad}
		If $d\ge 1,c\ge 2,$ $[\Si]$ is a non-torsion class and $M$ is an orientable, and $K$ is the infinite subset of $\Z_{\ge 2}$ defined in Lemma \ref{lemthom}, then for any two integer $k,l\in K$ with $k<l$ and any real numbers $\lam>1$,  there exist a non-empty open set $\Om'_{[\Si],k,l,\lam}$ in the space of Riemannian metrics, depending on $[\Si]$, $k,$ $l$, and $\lam,$ such that for any metric $h\in \Om'_{[\Si],k,l,\lam},$ we have
		\begin{align}
			\frac{
				\no{k[\Si]}_h^{\Z}}{k}>\lam\frac{\no{l[\Si]}_h^{\Z}}{l}.
		\end{align} 
	\end{cor}
	In other words, Federer's asymptotic (\ref{eqfed}) is sharp and cannot be improved to a finitary estimate on the growth of $\no{k[\Si]}_h^{\Z}$ for all $k\in\N$. An alternative way to say this is that the homogeneity of the norm $\no{\cdot}_h^{\Z}$ fails tremendously. One should also compare with the results in \cite{FMmc,RY,BWmc}, where the authors are considering the failure of homogeneity of $\no{\cdot}_h^{\Z}$ of Plateau problems in Euclidean space. 
	
	In summary, the failure of homogeneity of $\no{\cdot}_{g}^{\Z}$ means we cannot control the growth of $\no{\cdot}_{g}^{\Z}$ on $H_d(M,\Z)$ without information from $H_d(M,\R)$. 
	\subsection{Why real homology can be circumvented in the Euclidean setting?}
	In Section \ref{secd=1}, we provide an alternative proof of Theorem \ref{thmcm} that did not use infomation from real homology groups $H_d(\R^{d+c},W,\R)$.
	
The author's interpretation is that analogous ill-behaved phenomenons like Corollary \ref{corbad}	cannot happen in the setting of relative homology classes  $[\Si]\in H_d(\R^{d+c},W,\R).$ See related conjectures in \cite{RY}. \subsection{More on when $\tau_d$ does not divide $n$}
	The reader might wonder what more we can say beyond Theorem \ref{thmvg} when $\tau_d\not|n$. Recall Fact \ref{fctfrmn}. Without the condition of $\tau_d|n$, we can no longer make sure that Federer's ``representative" modulo $n$ of elements in $\mn_g^{\Z/n\Z}\left([\Si]\right)$ still lies in $[\Si]$ in integral homology, thus resulting in the weaker Theorem \ref{thmvg}. Indeed, we can construct examples to show that this is not true when $\tau_d>1$ and $d<c$ using a stronger version of Lemma \ref{lemzhang}. However, for general $n$ we can instead focus on the $\Z/n\Z$ homology and say something more than Corollary \ref{coras}.
	\begin{cor}We have
		\begin{align}\label{eqliminf}
			\liminf_{\substack{n\to\infty\\ n\in\N}}\frac{\#\{\om|\textnormal{all $\Z/n\Z$-area-minimizing representatives of }\om\in  H_d(M,\Z/n\Z)\textnormal{ are $\Z$-cycles}\}}{\#\{\om|\om\in H_d(M,\Z/n\Z)\textnormal{ admitting $\Z$-representatives}\}}>0.
		\end{align}
	\end{cor}
	\begin{proof}
		The denominator of (\ref{eqliminf}) equals 
		\begin{align}\label{eqden}
			\# \left(H_d(M,\Z)\right)^{\Z/n\Z}=n^{b_d}\prod_{i\in I}\operatorname{gcd}(p_i^{\nu_i},n).
		\end{align}
		Using the notation of Lemma \ref{lemmon}, Subsection \ref{subsectionfed} and Fact \ref{fctsm} show that whenever
		\begin{align*}
			C_g\no{[\Si]}_g^{\Z/n\Z}<\frac{n}{2},
		\end{align*}
		we have
		\begin{align}\label{eqmnc}
			\mn_g^{\Z/n\Z}\left([\Si]\right)\s\bigcup_{[\Pi]\in H_d(M,\Z)}\left(\mn_{g}^{\Z}\left([\Pi]\right)\right)^{\Z/n\Z}.
		\end{align}Expand $[\Si]$ with respect to the basis in Assumption \ref{decompbs}, we have
		\begin{align}
			[\Si]=\sum_{j=1}^{b_d}a_jf_j+\sum_{i\in I}b_iv_i,
		\end{align}for $\bigcup_{j=1}^{b_d}\{a_j\}\s \N,$ and $\bigcup_{i\in I}\{b_i\}\s \Z\cap (-\frac{n}{2},\frac{n}{2}].$ Apply Lemma \ref{lemmsbd} and Fact \ref{fctfc}, we obtain
		\begin{align*}
			\no{[\Si]}_g^{\Z}\le \upsilon(g)\nog{[\Si]}_g\le\upsilon(g)b_d\max_{1\le j\le b_d}|a_j|\max_{1\le j\le b_d}\nog{f_j}_g.
		\end{align*}
		This implies whenever
		\begin{align*}
			\max_{1\le j\le b_d}|a_j|<\frac{n}{2C_g\upsilon(g)b_d\max_{1\le j\le b_d}\nog{f_j}_g},
		\end{align*}
		equation (\ref{eqmnc}) holds. This implies that the numerator of (\ref{eqliminf}) is at least
		\begin{align}\label{innnom}
			\left(\frac{n}{2C_g\upsilon(g)b_d\max_{1\le j\le b_d}\nog{f_j}_g}\right)^{b_d}.
		\end{align}
		Combining (\ref{eqden}) and (\ref{innnom}), we are done.
	\end{proof}
	\begin{conj}
		We conjecture that the value of the left-hand side of (\ref{eqliminf}) is $1.$
	\end{conj}

	\subsection{Lavrentiev gaps}
	Another way to think of our main theorems and corollaries from the variational perspective is to view them as the manifestation of Lavrentiev gap phenomena among the norms $\no{\cdot}_g^{R}$ with $R=\Z,\R,\Z/n\Z$ for $n\in \Z_{\ge 2}.$ 
	\begin{defn}\label{defnlavgap}
		Define 
		\begin{align}\label{lavgap}
			\lag([\Si],R)=\{g|\no{[\Si]}_{g}^{\Z}>\no{[\Si]}_g^{R}\}.
		\end{align}
	\end{defn}
	By Fact \ref{fctnr}, we have
	\begin{fact}
		$\lag([\Si],R)$ is always a (possibly empty) open set.  
	\end{fact}
	By Fact \ref{fctcal},  $\lag([\Si],\R)$ has the geometric meaning  that
	\begin{fact}
		$\lag([\Si],\R)$ is the set of metrics $g$ where no element of $\mn_{g}^{\Z}([\Si])$ can be calibrated.
	\end{fact}
	
	Theorem \ref{thmv} and Theorem \ref{thmvn} translate to\begin{itemize}
		\item the infinite intersection $\cap_{n}\lag([\Si],\Z/n\Z)\cp$ for $\tau_d|n$ and $n$ large enough, is a closed set with non-empty interior,
		\item if $c\ge 2$, $[\Si]$ is free, and $M$ is orientable, then for $\operatorname{gcd}\left(n,\mu_{[\Si]}\right)=1$, the open set $\lag([\Si],\Z/n\Z)$ is non-empty.
		\item if $c\ge 2$, $[\Si]$ is free, and $M$ is orientable, then the open set $\lag([\Si],\R)$ is non-empty.
	\end{itemize}
	Using this formulation, one can raise many different problems, for instance
	\begin{conj}
		(Camillo De Lellis) The open set $\lag([\Si],\R)$ is dense in the space of smooth Riemannian metrics.
	\end{conj}A direct corollary of the above conjecture of De Lellis is that $\Z$-area-minimizing submanifolds are generically not calibrated. 
	\subsection{Naturalness of Theorem \ref{thmvn}}
	We want to remark that $\lag([\Si],\R)$ might appear naturally even when we only consider classical Riemannian manifolds with canonical metrics. For instance, we have
	\begin{cor}\label{corcp}
		If $[\Si]\in H_d(\mathbb{CP}^{\frac{d+c}{2}},\Z),$ the Fubini-Study metric lies in the closure of $\lag([\Si],\R)$.
	\end{cor}
	Here necessarily both $d,c$ are even and at least $2.$
	\begin{proof}
		The class $2[\Si]$ always admits a connected embedded smooth complex algebraic subvariety $N$ as a representative. For instance, one can take Fermat hypersurfaces in an equator $\mathbb{CP}^{\frac{d}{2}+1}$ of $\mathbb{CP}^{\frac{d+c}{2}}.$ Then by \cite[page 70 Example I]{HL}, $N$ is $\Z$-area-minimizing in the Fubini-Study metric $g_{FS}$ and calibrated by the form $\frac{1}{(d/2)!}\om^d,$ where $\om$ is the standard Kahler form in the Fubini-Study metric. Let $\be$ be a smooth function on $\mathbb{CP}^{\frac{d}{2}+1}$ such that $\be\ge 1$ and $\be\m(1)=N.$ Since $\be g_{FS}\ge g_{FS}$ as quadratic forms, by the definition of comass in (\ref{defncms}), in metric $\be g_{FS},$ the form $\frac{1}{(d/2)!}\om^d$ still calibrates $N$. This implies that 
		\begin{align}\label{eqsfs}
			\no{[\Si]}_{\be g_{FS}}^{\R}=\frac{1}{2}\no{2[\Si]}_{\be g_{FS}}^{\R}=\frac{1}{2}\ms_{\be g_{FS}}^{\Z}(N).
		\end{align}  Since $\frac{1}{(d/2)!}\om^d$ have comass $1$ in $\be g_{FS}$ only on $N,$ by the constancy theorem \cite[Section 4.9]{FMgmt}, any integral current calibrated by $\frac{1}{(d/2)!}\om^d$ in $\be g_{FS}$ must be an integer multiple of $N$. As $[\Si]$ is not a torsion class, we deduce that 
		\begin{fact}\label{fctnc}
			$N$ is the unique $\Z$-rectifiable current in $2[\Si]$ calibrated by $\frac{1}{(d/2)!}\om^d$.
		\end{fact} If $\no{[\Si]}^{\R}_{\be g_{FS}}=\no{[\Si]}^{\Z}_{\be g_{FS}}$, then by (\ref{eqsfs}), any element $T\in\mn_{\be g_{FS}}^{\Z}([\Si])$ must satisfy $\ms_{\be g_{FS}}^{\Z}(T)=\frac{1}{2}\ms_{\be g_{FS}}^{\Z}(N).$ Thus, we have $2T\in 2[\Si]$ and $2T$ has the same area as $N.$ By Fact \ref{fcthl}, this implies that $2T$ is also calibrated by $\frac{1}{(d/2)!}\om^d$, which contradicts Fact \ref{fctnc}. Replacing $\be$ by $1+\e(\be-1)$ for positive $\e\to0$. We are done.  
	\end{proof}
	Using Almgren's constructions in \cite[Section 5.11]{HF}, similar results to Corollary \ref{corcp} also hold for the standard flat tori $\R^{d+c}/\Z^{d+c}.$
\subsection{Extension of Result \ref{rstc}}
Result \ref{rstc} is stated for area-minimizing cones with subanalytic link. With more careful analysis using o-minimal geometry and the boundary case of the extended monotonicity formula, Result \ref{rstc} holds for area-minimizing integral currents with globally subanalytic support, i.e., $\ran$-definable support in o-minimal geometry \cite{MComin,VDDom}. Unfortunately, that will significantly prolong this manuscript and does not seem like the sharp condition.

Indeed, there exists non-globally-subanalytic area-minimizing integral currents. Zhihan Wang (personal communications) points out that Hardt-Simon foliations of Lawson-Simons \cite{HS} are known to be not globally subanalytic in some cases. However, by \cite{FMeu}, each leaf of these folations are $\Z/n\Z$-area-minimizing for $n\ge 2$.

 Start with the following observations
\begin{fact}\cite[3.3 Lemma]{FMvbd}
	A $d$-dimensional $\Z/n\Z$-area-minimizing $\Z/n\Z$-flat chain in $\R^{d+c}$ has Euclidean volume growth and has a tangent cone at infinity.
\end{fact}
Here by Euclidean volume growth we mean that the mass of the chain in the ball of radius $r$ centered at $0$ is $O(r^d)$. However, there are many area-minimizing integral currents without Euclidean volume growth. By \cite{EWa,WSa}, any non-algebraic holomorphic subvariety in complex Euclidean space does not have Euclidean volume growth.

Thus, any analogue of Result \ref{rstc} must have Euclidean volume growth as the assumption. We conjecture that
\begin{conj}
If a $d$-dimensional $\Z$-area-minimizing integral current $T$ in $\R^{d+c}$ has Euclidean volume growth, then the $\Z/n\Z$-flat chain $T^{\Z/n\Z}$ is $\Z/n\Z$-area-minimizing for $n$ large enough.	
\end{conj} 
\subsection{Classifications of area-minimizing pairs of planes mod $2$ and mod $3$}\label{secremp}
The resolution of angle conjecture \cite{DN,GLac} classifies  $\Z$- and $\R$-area-minimizing pairs of planes. Theorem \ref{thmcpm} finishes the classification of $\Z/n\Z$-area-minimizing pairs of planes. The only cases left are $\Z/2\Z$ and $\Z/3\Z$ cases. Here we recall again Morgan's conjectures:
\begin{conj}\cite[Problem 16]{FMwtc}A pairs of oriented planes is $\Z/2\Z$-area-minimizing if and only if the pair is $\Z$-area-minimizing with any orientations on $P,Q$. 
\end{conj}
Partial progress is obtained by Lawlor in \cite{GLds,GL,GLmk}. The $\Z/3\Z$ case is closely related to size-minimization \cite{FMsm}. We have the following conjecture by Morgan (personal communications) :
\begin{conj}(Morgan)
A pairs of planes is $\Z/3\Z$-area-minimizing if and only if the pair is $\Z$-area-minimizing and the sum of characterizing angles is at least $\frac{2}{3}\pi.$
\end{conj}
Partial progress is obtained by Lawlor in \cite{GLpp,GL}.
\subsection{Further analogies with number theory}
\label{secdot}
Recall the definition of $p$-adic norms  \cite[Chapter I]{NKpa}. For readers with experience in number theory, the first question that comes to mind might be why we are not using $R=\Z_p$, i.e., the $p$-adic integers, as the Hasse principle is usually stated via the $p$-adic integers to simultaneously capture information from all solutions modulo all prime powers. Roughly speaking, the reason is that we now have both algebraic and analytic structures, and we need both structures to be compatible across different coefficient rings. 

By Ostrowski's theorem \cite[Chapter I]{NKpa}, up to equivalence, on $\Q$ there are only three types of norms that respect multiplication,
\begin{itemize}
	\item the absolute value norm,
	\item the $p$-adic norms,
	\item the trivial norm.
\end{itemize}	
Here by respecting multiplication, we mean a norm $\no{\cdot}$ satisfying $\no{pq}=\no{p}\no{q}$ for all $p,q\in\Q$. The trivial norm is defined to be $0$ on the zero element of a ring $R$ and $1$ otherwise. 

The author's heuristic is that for each norm in Ostrowski's theorem, we have a slightly different version of the Hasse principle.

By Definition \ref{defnnorm}, $\Z$-, $\R$- and $\Z/n\Z$-flat chains are defined using the absolute value norm and its induced norms on $\Z/n\Z.$  Our theorems can be interpreted as the Hasse principle with respect to the absolute value norm. There are two more canonical norms by Ostrowski's theorem.
\subsubsection{$p$-adic norms}
By \cite[Chapter II.1.1]{JSca}, $\Z_p$ is the inverse limit of $\Z/p^n\Z$ with natural numbers $n\to\infty$ and the system of homomorphisms being mod $p^n$ reductions. On one hand, this inverse limit system is not compatible with the norm $\no{\cdot}^{\Z/p^n\Z}$ in Definition \ref{defnnorm}. On the other hand, the $p$-adic norm does naturally induce a norm in the sense of Definition \ref{defnnorm} on $\Z/p^n\Z$, i.e., no longer respecting multiplication. The norm is defined by applying the $p$-adic norm to the unique lifts from $\Z/p^n\Z$ to $ \Z\cap[0,p^n-1),$ which we will also call the $p$-adic norms on $\Z/p^n\Z$. We conjecture that
\begin{conj}
	(The Hasse principle for $p$-adic variational problems) For a variational problem that can be formulated on chain spaces over $\Q_p,\Z_p$ and $\Z/p^n\Z$ for $n\ge 1$, all equipped with the  $p$-adic norm, the solutions in $\Z_p$-chains can be reconstructed from the corresponding solutions in $\Q_p$-chains and $\Z/p^n\Z$-chains for all $n\ge 1$.  
\end{conj} 
Here $\Q_p$ is the $p$-adic rational numbers, the fraction field of $\Z_p$. The choice of $\Q_p$ is to mimic the role of $\R$ as $\R$ is the completion of the fraction field $\Q$ of $\Z$ with respect to the absolute value norm.
\subsubsection{Asymptotic behavior of $\Z/n\Z$-area-minimizing representatives as $n\to\infty$}
If we want a limit of $\Z/n\Z$ compatible with $\no{\cdot}^{\Z/n\Z}$ as $n\to\infty$, we can renormalize the norm on $\Z/n\Z$ to $\frac{1}{n}\no{\cdot}^{\Z/n\Z},$ and take the direct limit of $\Z/n\Z$ with the homomorphisms being $\times m$ multiplications from $\Z/n\Z$ to $\Z/nm\Z.$ The resulting direct limit of $\Z/n\Z$ as $n\to\infty$ is $\Q/\Z$ and the norm is the absolute value norm applied to the unique lifts from $\Q/\Z$ to $[-\frac{1}{2},\frac{1}{2}).$ Recall Federer's result (\ref{eqfed}) and Section \ref{secdrh}. We conjecture that
\begin{conj}
	$\R/\Z$-area-minimizing flat chains represents the asymptotic behavior of $\Z/n\Z$-area-minimizing flat chains as $n\to\infty$, just as $\R$-area-minimizing flat chains represents the asymptotic behavior of $\Z$-area-minimizing flat chains.
\end{conj}
Here we use $\R/\Z$ instead of $\Q/\Z$ as we want our ring to be complete with respect to our choices of norms.
\subsubsection{the trivial norm}
By \cite{BWrc,BWdt} the trivial norm induces the so-called size, i.e., Hausdorff measure of rectifiable underlying sets of chains. We will follow the literature convention and call area-minimizing flat chains with respect to the trivial norm size-minimizing flat chains. 

Unfortunately, $\Z$-flat chains with respect to the trivial norm do not enjoy compactness theorems, so finding size-minimizing representatives in integral homology is a long-standing open problem. The only progress to date is by Frank Morgan \cite{FMsm} in some subcases of codimension $c=1$. On the other hand, by \cite{BWrc,BWdt},  $\Z/n\Z$-size-minimizing flat chains always exist and they enjoy good regularity:
\begin{fact}
	A $d$-dimensional $\Z/n\Z$-size-minimizing flat chain $T$ satisfies
	\begin{itemize}
		\item  (\cite{BWfrfc}) the $d$-th stratum in the Almgren stratification of $T$ is a smooth open not necessarily connected $d$-dimensional submanifold,
		\item (\cite{FMmcr,LScy}) near each point in the $(d-1)$-th stratum in the Almgren stratification of $T$, the support of $T$ decomposes into three $d$-dimensional submanifolds with boundaries meeting along a $(d-1)$-dimensional submanifold at equal angles of $\frac{2\pi}{3}$, i.e., forming $Y$-shaped triple junctions.
	\end{itemize} 
\end{fact}
We conjecture that
\begin{conj}
	(The Hasse principle for size-minimizing flat chains) $\Z$-size-minimizing representatives of an integral homology class $[\Si]$ exist and can be reconstructed from $\R$-size-minimzing representatives of $[\Si]^{\R}$ and $\Z/n\Z$-size-minimizing representatives of $[\Si]^{\Z/n\Z}$ for $n\in\Z_{\ge 2}$.
\end{conj}

	\printbibliography
	\end{document}